\documentclass[11pt]{amsart}
\usepackage{amscd,amsmath,amssymb,amsfonts}
\usepackage[cmtip, all]{xy}

\newtheorem{thm}{Theorem}[section]
\newtheorem{prop}[thm]{Proposition}
\newtheorem{lem}[thm]{Lemma}
\newtheorem{cor}[thm]{Corollary}

\theoremstyle{definition}

\newtheorem{defn}[thm]{Definition}
\theoremstyle{remark}
\newtheorem{remk}[thm]{Remark}
\newtheorem{remks}[thm]{Remarks}

\newtheorem{exm}[thm]{Example}
\newtheorem{exms}[thm]{Examples}
\newtheorem{notat}[thm]{Notation}
\numberwithin{equation}{section}

{\hfill$\square$\end{defn}}
{\hfill$\square$\end{remk}}
{\hfill$\square$\end{remks}}
{\hfill$\square$\end{exm}}
{\hfill$\square$\end{exms}}
{\hfill$\square$\end{notat}}

\newcommand{\thmref}{Theorem~\ref}

\newcommand{\lemref}{Lemma~\ref}

\newcommand{\sC}{{\mathcal C}}
\newcommand{\sD}{{\mathcal D}}

\newcommand{\sK}{{\mathcal K}}

\newcommand{\sO}{{\mathcal O}}

\newcommand{\sT}{{\mathcal T}}

\newcommand{\sW}{{\mathcal W}}

\newcommand{\A}{{\mathbb A}}

\renewcommand{\H}{{\mathbb H}}

\renewcommand{\P}{{\mathbb P}}

\newcommand{\CH}{{\rm CH}}

\newcommand{\inj}{\hookrightarrow}
\newcommand{\red}{{\rm red}}
\newcommand{\codim}{{\rm codim}}

\newcommand{\Hom}{{\rm Hom}}

\newcommand{\Spec}{{\rm Spec \,}}

\newcommand{\Sch}{{\operatorname{\mathbf{Sch}}}}

\newcommand{\Sm}{{\mathbf{Sm}}}
\newcommand{\SmProj}{{\mathbf{SmProj}}}

\newcommand{\Ab}{{\mathbf{Ab}}}

\newcommand{\End}{{\operatorname{\text{End}}}}

\newcommand{\ds}{{/\kern-3pt/}}

\newcommand{\Supp}{{\operatorname{Supp}}}

\newcommand{\Cor}{{\operatorname{Cor}}}

\renewcommand{\TH}{{\operatorname{TCH}}}
\newcommand{\un}{\underline}
\newcommand{\ov}{\overline}

\newcommand{\dgn}{{\operatorname{degn}}}
\renewcommand{\dim}{\text{\rm dim}}

\newcommand{\tuborg}{\left\{\begin{array}{ll}}
\newcommand{\sluttuborg}{\end{array}\right.}

\newcommand{\mot}{{\rm Mot}}

\begin{document}
\title[Cycles with modulus]{A module structure and a vanishing theorem for cycles with modulus}
\author{Amalendu Krishna and Jinhyun Park}
\address{School of Mathematics, Tata Institute of Fundamental Research,  
1 Homi Bhabha Road, Colaba, Mumbai, India}
\email{amal@math.tifr.res.in}
\address{Department of Mathematical Sciences, KAIST, 291 Daehak-ro Yuseong-gu, 
Daejeon, 34141, Republic of Korea (South)}
\email{jinhyun@mathsci.kaist.ac.kr; jinhyun@kaist.edu}

\keywords{algebraic cycle, $K$-theory}

\subjclass[2010]{Primary 14C25; Secondary 19E15, 13F35}
\begin{abstract}
We show that the higher Chow groups with modulus of Binda-Kerz-Saito for a smooth quasi-projective scheme $X$ is a module over the Chow ring of $X$. From this, we deduce certain pull-backs, the projective bundle formula, and the blow-up formula for higher Chow groups with modulus.

We prove vanishing of $0$-cycles of higher Chow groups with modulus on various affine varieties of dimension at least two. This shows in particular that the multivariate analogue of Bloch-Esnault--R\"ulling computations of additive higher Chow groups of 0-cycles vanishes.
\end{abstract}

\maketitle


\section{Introduction}
Recently, algebraic cycles with certain constraints at infinity, called modulus conditions, are drawing attentions. These cycles with modulus originate from the work of S. Bloch and H. Esnault in \cite{BE1}, where the first such groups were defined. They computed the $0$-cycle groups to give a motivic interpretation of the absolute K\"ahler forms of a field. Notably the subject of additive higher Chow groups emerged from there and it was studied in \cite{KL}, \cite{KP}, \cite{KP3}, \cite{KP2}, \cite{P1}, \cite{P2} and \cite{R}. 

Continuing these, in 2010 the authors began to study a generalization, that we call \emph{multivariate additive higher Chow groups}, a glimpse of which appeared in \cite{KP2}. Meanwhile, F. Binda, M. Kerz and S. Saito in \cite{BS}, \cite{KS} defined more general objects, called \emph{higher Chow groups with modulus} $\CH^q (X|D, n)$ of a scheme $(X,D)$ with an effective Cartier divisor. 

This paper is a result of fitting the studies of multivariate additive higher Chow groups into this new environment. We prove the following results in this note.

\begin{thm}Let $X$ be a smooth quasi-projective scheme over a field $k$. Then, there is a cap product $\cap_X: \CH^r (X, n_1) \otimes \CH_s (X|D, n_2) \to \CH_{s-r}(X|D, n_1 + n_2)$. 
\end{thm}

From this result, we deduce certain pull-back maps (\S \ref{sec:mod pull-back}), the projective bundle formula and the blow-up formula (\S \ref{sec:mod proj bundle}) for the higher Chow groups with modulus.

\begin{thm} Assume $r \geq 2$. When $k = \ov{k}$, ${\rm char} (k) = 0$ and $D \subset \mathbb{A}^r$ is an effective Cartier divisor with $\deg (D_{\rm red}) \leq r$, then $\CH^{r+n} (\mathbb{A}^r |D, n) = 0$ for $n \ge 0$. 

When $k = \mathbb{F}_q$ or $\ov{\mathbb{F}}_q$ and $X$ is an affine variety of dimension $r$ with an effective Cartier divisor $D \subset X$, then $\CH^{r+n} (X|D, n) = 0$ for $n \geq 1$.
\end{thm}

The multivariate additive higher Chow groups are the higher Chow groups with modulus that are attached to $(X\times \mathbb{A}^r, D_{\un{m}})$, where $D_{\un{m}} = \{ t_1 ^{m_1} \cdots t_r ^{m_r} = 0 \}$ for $(t_1, \cdots, t_r) \in \mathbb{A}^r$ and $m_i \geq 1$. When $r=1$, the group of $0$-cycles $\CH^{1+n} (\mathbb{A}^1|D_{m+1}, n)$ is the group $\mathbb{W}_m \Omega_k^n$ of the big de Rham-Witt forms (see R\"ulling \cite{R}). When $r \geq 2$, we prove: 

\begin{thm}
For $r \ge 2$ and $n \ge 0$, we have $\CH^{r+n} (\mathbb{A}^r|D_{\un{m}}, n) = 0$. When $X$ is a $k$-scheme of dimension $r-2$ with an effective Cartier divisor $D$, then $\CH^{r+n} (X \times \mathbb{A}^2| D \times \mathbb{A}^2 + X \times D_{(m_1, m_2)}, n) = 0$. 
\end{thm}

One may regard the above as the cycle-theoretic counterpart for the vanishing of the $K$-group $T(\mathbb{G}_a, \mathbb{G}_a, \mathcal{F}_1, \cdots, \mathcal{F}_n)$ of reciprocity functors in Ivorra-R\"ulling \cite[Theorem 5.5.1]{IR}. For codimension $1$ cycles, we have the following partial results (see Theorems \ref{thm:codim 1 n=0}, \ref{thm:codim 1 n=1}):

\begin{thm}For $r \geq 2$, $\CH^1 (\mathbb{A}^r|D_{(1, \cdots, 1)}, 0) = 0$ and $\CH^1 (\mathbb{A}^r|D_{(1, \cdots, 1)}, 1) \not = 0$.
\end{thm}

Based on Hesselholt \cite{Hesselholt Nagoya} on $K$-groups of $(\mathbb{A}^2, \{t_1 t_2 = 0 \})$, we guess $\CH^q (\mathbb{A}^2|D_{(1,1)}, n) = 0$ for $n<2q-1$, which is an analogue of Beilinson-Soul\'e vanishing conjecture. We have verified it when $q = n+2$ and when $q=1$ and $n=0$. When $q=1$ and $n=1$, it does not satisfy $n<2q-1$, and the group is nontrivial.

\vskip .2cm
 
\noindent \textbf{Conventions : } A $k$-scheme means a separated scheme of finite type over a field $k$, and a $k$-variety is an integral $k$-scheme. Let $\Sch_k$ be the category of $k$-schemes. For any $S \in \Sch_k$, let $\Sm_S$ be the category of schemes smooth over $S$ and $\SmProj_S$ be the subcategory of schemes smooth and projective over $S$. 
The product $X \times Y$ usually means $X \times_k Y$, unless we specify otherwise.

\section{Cycle complex with modulus}\label{section:RMC}

For $\P^1  = {\rm Proj}_k(k[s_0, s_1])$, we use $y = {s_1}/{s_0}$ as its coordinate. Let $\square:= \P^1 \setminus \{1\}$.  For $n \geq 1$, let $(y_1, \cdots, y_n) \in \square^n$ be the coordinates. A face $F\subset \square^n$ means a closed subscheme defined by the set of equations of the form $\{y_{i_1} = \epsilon_{1}, \cdots , y_{i_m} = \epsilon_{m}\}$ for some $1 \leq i_1< \cdots < i_m \leq n$ and $\epsilon_j \in \{ 0, \infty \}$. Let $\ov{\square}:= \mathbb{P}^1$. A face of $\ov{\square}^n$ is the closure of a face in $\square^n$. For $1 \leq i \leq n$, let $F_{n,i} ^1 \subset \ov{\square}^n$ be the closed subscheme given by $\{ y_i = 1 \}$. Let $F_n ^1:= \sum_{i=1} ^n F_{n,i} ^1$, which is the cycle associated to the closed subscheme $\ov{\square}^n \setminus \square^n$. Let $\square^0 = \ov{\square}^0 := \Spec (k)$. Let $\iota_{n,i, \epsilon}: \square^{n-1} \inj \square^{n}$ be the obvious inclusion for $i \in \{ 1, \cdots, n \}$ and $\epsilon \in \{ 0, \infty \}$.

\subsection{Basic lemmas}

We discuss some background lemmas first. 

\begin{lem}[{\cite[Lemma~2.1]{KP}}]\label{lem:Pull-D}
Let $X$ be a normal variety and let $D_1$ and $D_2$ be effective Cartier divisors on $X$ such that $D_1 \ge D_2$ as Weil divisors. Let $Y \subset X$ be a closed subset which intersects $D_1$ and $D_2$ properly. Let $f: Y^N \to X$ be the composite of the inclusion and the normalization of $Y_{\rm red}$. Then $f^*(D_1) \ge f^*(D_2)$.
\end{lem}

\begin{lem}\label{lem:cancel}
Let $f: Y \to X$ be a dominant map of normal integral $k$-schemes. Let $D$ be a Cartier divisor on $X$ such that the generic points of $\Supp(D)$ are contained in $f(Y)$. Suppose that $f^*(D) \ge 0$ on $Y$. Then $D \ge 0$ on $X$.
\end{lem}

\begin{proof} It refines \cite[Lemma 3.2]{KL} and \cite[Lemma 2.2]{KP}. Localizing at the generic points of $\Supp(D)$, we may assume $X = \Spec(A)$, for a dvr $A$ essentially of finite type over $k$. The divisor $D$ is given by a rational function $a= u \pi^n$ in ${\rm Frac}(A)$, where $ u \in A^{\times}$, $n \in \mathbb{Z}$, and $\pi$ is a uniformizing parameter of $A$. By our assumption, for some $y \in Y$, $f(y)$ is the closed point of $X$. Let $U \subset Y$ be an affine open neighborhood of $y$. Here, $f^* (D)|_U \geq 0$ and replacing $Y$ by $U$, we may assume $Y$ is affine. Then, for some closed point $y$, its image $f(y)$ is the closed point of $X$, but $f$ is dominant, so $f: \Spec (\mathcal{O}_{Y, y}) \to X$ is surjective. By abuse of notations, let $f^* : A \to \mathcal{O}_{Y, y}$ be the corresponding $k$-algebra homomorphism. In particular, the image  $f^*(\pi) \in \mathcal{O}_{Y,y}$ of $\pi$ is nonzero in the maximal ideal $\mathcal{M}_{Y, y} \subset \mathcal{O}_{Y,y}$. That $f^* (D) \geq 0$ implies $f^* (a) \in \mathcal{O}_{Y,y}$. Since the image $f^* (u) \in \mathcal{O}_{Y,y}$ of the unit $u$ is also a unit and $\pi \in \mathcal{M}_{Y,y}$, the element $f^* (a) = f^* (u) f^* (\pi)^n $ can lie in $ \mathcal{O}_{Y,y}$ only when $ n \geq 0$. Thus, $D$ is effective.
 \end{proof}

\subsection{Cycles with modulus}\label{sect:CWM}

Let $X \in \Sch_k$. For effective Cartier divisors $D_1$ and $D_2$ on $X$, we say $D_1 \leq D_2$ if $D_1 + D= D_2$ for some effective Cartier divisor $D$ on $X$. A \emph{scheme with an effective divisor} (sed) is a pair $(X,D)$, where $X \in \Sch_k$ and $D$ an effective Cartier divisor. A morphism $f: (Y, E) \to (X, D)$ of seds is a morphism $f: Y \to X$ in $\Sch_k$ such that $f^*(D)$ is defined as a Cartier divisor on $Y$ and $f^* (D) \leq E$. In particular, $f^{-1} (D) \subset E$. 
If $f: Y \to X$ is a morphism of $k$-schemes, and $(X, D)$ is a sed such that $f^{-1} (D ) = \emptyset$, then $f: (Y , \emptyset) \to (X, D)$ is a morphism of seds. 

\begin{defn}[{\cite{BS}, \cite{KS}}]\label{defn:modulus}
Let $(X,D)$ and $(\ov{Y},E)$ be schemes with effective divisors. Let $Y = \ov{Y} \setminus E$. Let $V \subset X \times Y$ be an integral closed subscheme with closure $\ov{V} \subset X \times \ov{Y}$. We say $V$ is a \emph{has modulus $D$ (relative to $E$)} if $\nu^*_V(D \times \ov{Y}) \le \nu^*_V(X \times E)$ on $\ov{V}^N$, where $\nu_V: \ov{V}^N \to \ov{V} \inj X \times \ov{Y}$ is the normalization followed by the closed immersion.
\end{defn}  

In case $Y= \ov{Y} = \Spec(k)$, that $V$ has modulus $D$ on $X \times Y$ is equivalent to $V \cap D = \emptyset$. We now state the following version of the containment lemma \cite[Proposition 2.4]{KP}, whose proof is almost identical so we omit it.

\begin{prop}[Containment lemma]\label{prop:CL*}
Let $(X,D)$ and $(\ov{Y}, E)$ be schemes with effective divisors and $Y= \ov{Y} \setminus E$. If $V \subset X \times Y$ is a closed subscheme with modulus $D$ relative to $E$, then any closed subscheme $W \subset V$ has modulus $D$ relative to $E$, too.
\end{prop}

\begin{defn}[{\cite{BS}, \cite{KS}}]\label{defn:partial complex}Let $(X,D)$ be a scheme with an effective divisor. For $r \in \mathbb{Z}$ and $n \geq 0$, let $\un{z}_r (X|D, n)$ be the free abelian group on integral closed subschemes $V \subset X \times \square^n$ of dimension $r+n$ satisfying the following conditions:

\begin{enumerate}
\item (Face condition) for each face $F \subset \square^n$, $V$ intersects $X \times F$ properly. 
\item (Modulus condition) $V$ has modulus $D$ relative to $F_n ^1$ on $X \times \square^n$. 
\end{enumerate}
\end{defn}

We usually drop the phrase ``relative to $F_n^1$'' for simplicity. A cycle in $\un{z}_r (X|D, n)$ is called an \emph{admissible cycle with modulus $D$}. Using Proposition~\ref{prop:CL*}, one checks that if $V$ has modulus $D$ on $X \times \square^n$, then $({\rm Id}_X \times \iota_F)^* (V)$ has modulus $D$ on $X \times F \simeq X \times \square^d$, where $\iota_F: F \hookrightarrow \square^n$ and $d = \dim F$ (see \cite[Lemma 2.4]{BS}). We deduce that $(n \mapsto \un{z}_r(X|D,n))$ is a cubical abelian group. In particular, the groups $\un{z}_r (X|D, n)$ form a complex with the boundary map $\partial= \sum_{i=1} ^n (-1)^i (\partial_i ^{\infty} - \partial_i ^0)$, where $\partial_i ^{\epsilon} = \iota_{n, i, \epsilon} ^*$.

\begin{defn}[{\cite{BS}, \cite{KS}}]
The complex $(z_r(X|D, \bullet), \partial)$ is the nondegenerate complex associated to $(n\mapsto \un{z}_r(X|D,n))$, i.e., $z_r(X|D, n): = {\un{z}_r(X|D, n)}/{\un{z}_r(X|D, n)_{\dgn}}.$ The homology $\CH_r(X|D, n): = H_n (z_r(X|D, \bullet))$ for $n \geq 0$ is called \emph{higher Chow group} of $X$ with modulus $D$. If $X$ is equidimensional of dimension $d$, for $q \geq 0$, we write $\CH^q(X|D, n) = \CH_{d-q}(X|D, n)$.
\end{defn}

\begin{remk}\label{remk:RMC-MC}When $D= \emptyset$, this is the cubical higher Chow group of \cite{Bl1}, while if $X= Y \times \mathbb{A}^1$ with $D= \{ t^{m+1}= 0 \}$, where $t \in \mathbb{A}^1$, this is the additive higher Chow group of $Y$ with modulus $m$ of \cite{BE1}, \cite{P2}, \cite{R}. If $D_2 \ge D_1$ are two effective Cartier divisors on $X$, there is a canonical inclusion $z_r(X| D_2, \bullet) \inj {z}_r(X| D_1, \bullet)$, thus a canonical map $\CH_r(X|D_2, n) \to \CH_r(X|D_1, n)$. In particular, since $\emptyset \le D$, we have $\CH_r(X|D, n) \to \CH_r(X, n).$
\end{remk}

\subsection{Functorial properties}\label{sect:PPFP} We say that a morphism of schemes with effective divisors is \emph{proper} or \emph{flat}, if the underlying morphism of $k$-schemes is proper or flat.

\begin{lem}\label{lem:pre pro image} Let $f: (Y,E) \to (X,D)$ be a proper morphism of schemes with effective divisors. Let $Z \subset Y \times \square^n$ be a closed irreducible subscheme with modulus $E$, and let $W:= f(Z) \subset X \times \square^n$. Then, $W$ has modulus $D$.
\end{lem}

\begin{proof} Let $\ov{Z} \subset Y \times \ov{\square}^n$ and $\ov{W} \subset X \times \ov{\square}^n$ be the Zariski closures, and let $\nu_Z: \ov{Z}^N \to Y \times \ov{\square}^n$ and $\nu_W: \ov{W}^N \to X \times \ov{\square}^n$ be the normalizations of $\ov{Z}$ and $\ov{W}$, composed with the closed immersions, respectively. Since $\ov{Z} \to \ov{W}$ is dominant, the universal property of normalization gives a morphism $h: \ov{Z}^N \to \ov{W}^N$ such that $(f \times {\rm Id}_{\ov{\square} ^n}) \circ \nu_Z = \nu_W \circ h$. This gives the identities $h^* \nu_W^* (D \times \ov{\square} ^n) = \nu_Z ^* (f^* (D) \times \ov{\square}^n)$ and $h^* \nu_W^* (X \times F_n ^1) = \nu_Z^* (Y \times F_n ^1)$. Hence, we get 
\[
h^* \nu_W ^* (D \times \ov{\square}^n) = \nu_Z ^* (f^* (D) \times \ov{\square}^n) \leq ^{\dagger} \nu_Z^* (E \times \ov{\square}^n) \leq ^{\ddagger} \nu_Z ^* (Y \times F_n^1) = h^* \nu_W ^* (X \times F_n ^1),
\]
where $\dagger$ is the definition of morphisms of schemes with effective divisors in \S \ref{sect:CWM} and $\ddagger$ holds since $Z$ has modulus $E$. Since $h$ is a generically finite proper morphism of normal integral schemes (in particular surjective), we deduce $\nu_W ^* (D \times \ov{\square}^n) \leq \nu_W^* (X \times F_n ^1)$  by Lemma \ref{lem:cancel}.\end{proof}

\begin{lem}\label{lem:projective image}Let $f: Y \to X$ be a proper morphism of quasi-projective $k$-varieties. Let $D \subset X$ be an effective Cartier divisor such that $f(Y) \not \subset D$. Let $Z \in z^q (Y|f^* (D), n)$ be an irreducible cycle. Let $W = f(Z)$ on $X \times \square^n$. Then, $W \in z^s (X|D, n)$, where $s= \codim_{X \times \square^n} (W)$. 
\end{lem}

\begin{proof}
It generalizes \cite[Proposition 5.2]{KP}. $W$ has modulus $D$ by Lemma \ref{lem:pre pro image}. We prove that $W$ intersects all faces properly. For codimension $1$ faces $F\subset \square^n$, note that $W$ intersects $X \times F$ properly if and only if $W\not \subset X \times F $. Suppose $W \subset X \times F$. 

Let $f_n= f \times {\rm Id}_{\square^n}$. Then, $Z \subset f_n ^{-1} (f_n (Z)) = f_n ^{-1} (W) \subset f_n ^{-1} (X \times F) = Y \times F$.  But $Z$ intersects $Y \times F$ properly, so $Z \not \subset Y \times F$, which is a contradiction. Hence, $W$ intersects $X \times F$ properly when $\codim _{\square^n} F = 1$. For higher codimensional faces, we apply induction on codimension of the given face, together with the above codimension $1$ case. Since $Z$ intersects all faces of any codimension properly, we deduce the same for $W$.
\end{proof}

\begin{prop}[Proper push-forward]\label{prop:PFF}
Let $f: (Y,E) \to (X,D)$ be a proper morphism of schemes with effective divisors. Then, it induces $f_*: {z}_r(Y|E, \bullet) \to {z}_r(X|D, \bullet)$ and $f_*: \CH_r(Y|E, n) \to \CH_r(X|D, n)$ such that $(f\circ g)_* = f_* \circ g_*$.
\end{prop} 

\begin{proof}For an irreducible $Z \subset Y \times \square^n$ with $W:= f(Z)$, we define $f_* ([Z]) := 0$ if $\dim (W) < \dim (Z)$, and $[k(Z):k(W)]\cdot [W]$ if $\dim (W) = \dim (Z)$ (see \cite{Fulton}). One checks immediately that $f_*$ respects the face condition by using the argument as in Lemma \ref{lem:projective image}, while the modulus condition for $W$ holds by Lemma \ref{lem:pre pro image}. That $(f\circ g)_* = f_* \circ g_*$ is immediate.
\end{proof}

The following case of proper push-forward will play an important role in \S \ref{section:0-cycles}.

\begin{cor}\label{cor:induced map}
Let $(X, D)$ be a scheme with an effective divisor and let $f: Y \to X$ be a proper map such that $f^{-1}(D) = \emptyset$. Then the push-forward map $f_*: \CH_r(Y, n) \to \CH_r(X, n)$ factors into $\CH_r(Y, n) \xrightarrow{\iota_*} \CH_r(X|D, n) \to \CH_r(X, n)$. 
\end{cor}

\begin{proof}
The map $f:(Y, \emptyset) \to (X,D)$ is a proper morphism of seds. So, the corollary follows from Proposition \ref{prop:PFF} because $\CH_r (Y|\emptyset, n) = \CH_r (Y, n)$. 
\end{proof}

\begin{prop}[Flat pull-back]\label{prop:FPB}
Let $f: Y \to X$ be a flat morphism of relative dimension $d$ and $D$ an effective Cartier divisor on $X$.
Then, it induces $f^*: {z}_r(X|D, \bullet) \to {z}_{d+r}(Y|f^*(D), \bullet)$ such that $(f \circ g)^* = g^* \circ f^*$.
\end{prop}

\begin{proof}For an integral admissible closed subscheme $Z \subset X \times \square^n$, we let $f^* ([Z])$ be the cycle associated to the scheme $f^{-1} (Z)$ in the sense of \cite{Fulton}. As in \cite{Bl1}, one checks that $f^*$ so defined respects the face condition. So, it remains to verify the modulus condition. For this, let $W$ be an irreducible component of $f^* (Z)$. Let $\ov{W} \subset Y \times \ov{\square}^n$ and $\ov{Z} \subset X \times \ov{\square}^n$ be the Zariski closures of $W$ and $Z$, and let $\nu_W : \ov{W}^N \to Y \times \ov{\square}^n$ and $\nu_Z : \ov{Z}^N \to X \times \ov{\square}^n$ be the normalizations of the Zariski closures composed with the closed immersions. The dominant map $W \to Z$ induces the map $h: \ov{W}^N \to \ov{Z}^N$ by the universal property of normalization, satisfying $(f \times {\rm Id}_{\ov{\square}^n} )\circ \nu_W = \nu_Z \circ h$. That $Z$ has modulus $D$ means $\nu_Z^* (D \times \ov{\square}^n) \leq \nu_Z ^* (X \times F_n ^1)$. Applying $h^*$ and using the above equality, we get $\nu_W ^* (f^* (D) \times \ov{\square}^n) \leq \nu_W ^* (Y \times F_{n} ^1)$ as desired, because $f^* (X \times F_n ^1) = Y \times F_n ^1$. That $(f \circ g)^* = g^* \circ f^*$ is obvious.
\end{proof}

Combined with \cite[Proposition 1.7]{Fulton}, we obtain: 

\begin{prop}\label{prop:pp}
Let $g: (X', D') \to (X, D)$ be a proper morphism of schemes with effective Cartier divisors, $f: Y \to X$ a flat morphism of schemes. Let $Y':= X' \times_X Y$ with the projections $f': Y' \to X'$ and $g': Y' \to Y$. Here $g'$ induces a proper morphism $(Y', {f'}^* (D')) \to (Y, f^* (D))$ of schemes with effective Cartier divisors.
 Then, $f^* g_* = g'_* f'^*$ as chain maps $z_r(X'|D, \bullet)\to  z_{r+d}(Y|f^*(D), \bullet)$.
\end{prop}

\section{Module structure over the Chow ring}\label{sec:module}
In this section, we prove that the higher Chow groups with modulus on smooth quasi-projective schemes are graded modules over the Chow ring. This result also improves \cite[Theorem~4.10]{KL}, where it was shown that the additive higher Chow groups are modules over the Chow ring of smooth projective varieties. The projectivity assumption in \emph{loc.cit.} was required because one needed a stronger moving result than Theorem \ref{thm:ML2} to take care of intersection of the closure of cycles with faces in $\ov{\square}^n$. We do not know if this stronger moving result works in the smooth quasi-projective case. We get around this issue in this note by combining Proposition \ref{prop:CL*} with some strategies in \cite[\S 4]{KL}. 

\subsection{Bloch's moving lemma}\label{section:BML}

Recall the following widely used notion:

\begin{defn}\label{defn:complex for moving}
Let $\mathcal{W}$ be a finite set of locally closed subsets of $X$ and let $e: \mathcal{W} \to \mathbb{Z}_{\geq 0}$ be a set function. Let $\un{z}^q (X, \bullet)$ denote the cycle complex of Bloch. Let $\un{z} ^q _{\mathcal{W}, e} (X, n) $ be the subgroup generated by integral cycles $Z \in \un{z}^q (X, n)$ such that for each $W \in \mathcal{W}$ and each face $F \subset \square^n$, we have $\codim _{W \times F} Z \cap (W \times F) \geq q - e(W)$. They form a subcomplex $\un{z}^q _{\mathcal{W},e}(X, \bullet)$ of $\un{z}^q (X, \bullet)$. Modding out degenerate cycles, we obtain the subcomplex $z^q _{\mathcal{W},e} (X, \bullet)\subset z^q (X, \bullet)$. We write $z^q _{\mathcal{W}} (X, \bullet) := z^q_{\mathcal{W}, 0} (X, \bullet)$.
\end{defn}

We use the following moving lemma of Bloch stated in \cite[Lemma 4.2]{Bl1}, where the localization in \emph{loc.cit.} is corrected in the main theorem of \cite{Bl2}, to construct module structure on higher Chow groups with modulus.

\begin{thm}[Bloch]\label{thm:ML2} 
Let $X$ be a smooth quasi-projective $k$-scheme. Let $\mathcal{W}$ be a finite set of locally closed subsets of $X$ and $e: \mathcal{W} \to \mathbb{Z}_{\geq 0}$ be a set-function. Then, the inclusion $z^q _{\mathcal{W}, e}(X, \bullet) \hookrightarrow z^q (X, \bullet)$ is a quasi-isomorphism.
\end{thm}

We use the following refined version that allows more flexibility for $\mathcal{W}$ (see \cite[p.112]{Hanamura}, \cite[Definition 2.1]{KL}, \cite[Definition 5.3]{KP3}). 

\begin{defn}\label{defn:refined moving}Let $X$ be a quasi-projective $k$-scheme, and let $T_1, \cdots, T_N$ be $k$-schemes. Let $\mathcal{W}$ be a finite collection of irreducible locally closed subsets $W_i \subset X \times T_i$. For each face $F \subset \square^n$, let $p_{F, i}: X \times F \times T_i \to X \times T_i$ be the projection. Let $\un{z}_{\mathcal{W}}^q (X, n)$ be the subgroup generated by integral cycles $Z \in \un{z} ^q (X, n)$, such that for each face $F \subset \square^n$, two sets $p_{F,i} ^{-1} (W_i)$ and $ (Z \cap (X \times F)) \times T_i$ intersect properly on $X \times F \times T_i$ for all $1 \leq i \leq N$. Modding out the degenerate cycles, we obtain $z_{\mathcal{W}}^q (X, \bullet)\subset z^q (X, \bullet)$. 
\end{defn}

\begin{lem}\label{lem:refined moving}Let $X$ be a quasi-projective $k$-scheme, and let $\mathcal{W}$ be as in Definition \ref{defn:refined moving}. Then, there exists a finite collection $\mathcal{C}$ of irreducible locally closed subsets of $X$ and a set function $e: \mathcal{C} \to \mathbb{Z}_{\geq 0}$ such that $z_{\mathcal{C}, e} ^q (X, \bullet) = z_{\mathcal{W}} ^q (X, \bullet)$, where the left one is as in Definition \ref{defn:complex for moving}. Furthermore, the inclusion map $z_{\mathcal{W}} ^q (X, \bullet) \hookrightarrow z ^q (X, \bullet)$ is a quasi-isomorphism if $X$ is smooth .
\end{lem}

\begin{proof}
For the first part, as in \cite[Proposition 2.2]{KL}, we define $C_{i,d}:= \{ x \in X | (x \times T_i) \cap W_i \mbox{ contains a component of dimension} \geq d \}$ for each $1 \leq i \leq N$. We write $C_{i, d} \setminus C_{i, d+1} = \bigcup_{j} C_{i,d} ^j$, a finite union of irreducible locally closed subsets. Let $\mathcal{C}= \{ C_{i, d} ^j \}$ and define $e: \mathcal{C} \to \mathbb{Z}_{\geq 0}$ by $e (C_{i,d} ^j) = \dim W_i - d - \dim C_{i,d} ^j $. One checks $z_{\mathcal{C}, e} ^q (X, \bullet) = z_{\mathcal{W}} ^q (X, \bullet)$. When $X$ is smooth, this complex is quasi-isomorphic to $z^q (X, \bullet)$ by Theorem \ref{thm:ML2}, so the second part holds.
\end{proof}

\begin{lem}\label{lem:distinguished complex}
Let $f: X \to Y$ be a morphism of quasi-projective $k$-schemes. Let $\mathcal{W}$ be a finite collection over $X$ as in Definition \ref{defn:refined moving}. Then, there exists a finite collection $\mathcal{W}'$ over $Y$ as in Definition \ref{defn:refined moving}, such that $f^*: z^q _{\mathcal{W}'} (Y, \bullet) \to z^q _{\mathcal{W}} (X, \bullet)$, given by taking the associated cycle of $f^{-1} (Z)$ for each cycle $Z$, is a well-defined chain map.
\end{lem}

\begin{proof}Let $\mathcal{W}$ consist of $W_i \subset X \times T_i$ for $1 \leq i \leq N$. Define $\mathcal{W}'$ to be the collection of the sets $Y \times W_i \subset Y \times T_i'$, with $T_i ' = X \times T_i$ for $1 \leq i \leq N$, and the transpose of the graph cycle, ${}^t\Gamma_f \subset Y \times T'_{N+1}$ with $T_{N+1}' = X$. 
Since $\mathcal{W}'$ contains ${}^t \Gamma_f$, one checks that $f^*$ is well-defined, and one easily sees that $f^* (z_{\mathcal{W}'}^q (Y, \bullet))\subset z_{\mathcal{W}} ^q (X, \bullet)$.
\end{proof}

\subsection{External action of Chow cycles}
Let $X, Y \in \Sch_k$ and let $(Y,D)$ be a scheme with an effective divisor. Set $D_X:= X \times D$.

\begin{lem}\label{lem:Ext-Mod}
Let $Z \in \un{z}_r(X,n_1)$ and $W \in \un{z}_s(Y|D, n_2)$ be cycles. Then $\tau(Z \times W)$ on $ X \times Y \times \square^{n_1+n_2}$ is an element of $\un{z}_{r+s}(X \times Y|D_X, n_1 + n_2)$, where $\tau: X \times \square^{n_1} \times Y \times \square^{n_2} \xrightarrow{\simeq} X \times Y \times \square^{n_1 + n_2}$ is the obvious exchange of factors.
\end{lem}

\begin{proof}We may assume $Z$ and $W$ are irreducible as the general case immediately follows by extending the result $\mathbb{Z}$-bilinearly. Let $V$ be an irreducible component of $ \tau(Z \times W)$. The face condition for $V$ is immediate. For the modulus condition, let $\ov{V}$ and $\ov{W}$ be the Zariski closures of $V$ and $W$ in $X \times Y \times \ov{\square}^{n_1 + n_2 }$ and $Y \times \ov{\square}^n_2$, respectively. Consider the commutative diagram, 
\begin{equation}\label{eqn:Ext-P-0}
\xymatrix{
\ov{V} ^N \ar[r]^{\nu_V} \ar[d]_{q_1} & \ov{V} 
\ar[r]^{\iota_{V} \ \ \ \ \ \ \ } \ar[d]_{q_2} & X \times Y \times \ov{\square}^{n_1 + n_2} 
\ar[d]^{q_3} \\
\ov{W}^N \ar[r]^{\nu_W} & \ov{W} \ar[r]^ {\iota_W} & Y \times \ov{\square}^{n_2},}
\end{equation}
where $\nu_V$ and $ \nu_W$ are normalizations, $\iota_V$ and $\iota_W$ are closed immersions, $q_3$ is the projection, $q_2$ its restriction, and $q_1$ is induced by the universal property of normalization. Note that $q_3 ^* (D \times \ov{\square}^{n_2}) = D_X \times \ov{\square}^{n_1 + n_2}$ and $q_3 ^* (Y \times F_{n_2} ^1) \leq X \times Y \times F_{n_1 + n_2} ^1$. That $W$ has modulus $D$ means $(\iota_W \circ \nu_W)^*(D \times \ov{\square}^{n_2}) \leq (\iota_W \circ \nu_W)^*(Y \times F_{n_2} ^1)$. Applying $q_1 ^*$ and using the commutativity, we get 
\[
(\iota_V \circ \nu_V)^*(D_X \times \ov{\square}^{n_1 + n_2}) \leq (\iota_V \circ \nu_V)^* \circ q_3 ^* (Y \times F_{n_2} ^1) 
\leq (\iota_V \circ \nu_V)^* (X \times Y \times F_{n_1 + n_2} ^1),
\] 
which shows that $V$ has modulus $D_X$. 
\end{proof}

Thus, we have $\boxtimes_{n_1, n_2}:\un{z}_r(X,n_1)\otimes\un{z}_s(Y| D,n_2)\to \un{z}_{r+s}(X\times Y |D_X, n_1 + n_2),$ given by $\boxtimes_{n_1, n_2} (Z \otimes W) = Z \boxtimes W= \tau_* (Z \times W).$ A straightforward computation of the boundaries of $Z\boxtimes W$ yields:

\begin{prop}\label{prop:Ext-Mod-0}
There is an external product $\boxtimes:\CH_r(X,n_1)\otimes \CH_s(Y| D,n_2)\to \CH_{r+s}(X\times Y | D_X, n_1 + n_2),$ compatible with flat pull-back and proper push-forward.
\end{prop}

\subsection{Cap product}

Our next goal is to construct a cap product $\cap_X: \CH^r(X, n_1) \otimes \CH_s(X| D, n_2)\to \CH_{s-r}(X|D, n_1 + n_2 )$, where $X$ is smooth quasi-projective with an effective Cartier divisor $D$. Consider $\Delta_{X,n}= \Delta_X \times {\rm Id}_{\square^n} : X \times \square^n \to X \times X \times \square^n$, where $\Delta_X$ is the diagonal embedding. As before, let $D_X=X \times D$. 

\begin{defn}\label{defn:good-int}
Let $\un{z}^r(X\times X|D_X, n)_\Delta$ be the group generated by integral cycles $Z \in \un{z}^r(X\times X|D_X, n)$ such that (i) $\codim_{X\times F}(\Delta^{-1}_{X,n}(Z)\cap (X\times F)  )\ge r$ for all faces $F$ of $\square^n$, and (ii) $\Delta^*_{X,n}(Z) \in z^r(X|D, n)$.
\end{defn}

The subgroups $\un{z}^r(X\times X| D_X,n)_\Delta$ form a cubical subgroup $\un{z}^r(X\times X|D_X, -)_\Delta\subset \un{z}^r(X\times X|D_X, -)$ and one checks the maps $\Delta^*_{X,n}$ give a well-defined chain map $\Delta_X^*:z^r(X\times X|D_X, \bullet)_\Delta\to z^r(X|D, \bullet).$

\begin{defn}Suppose $\alpha \in z^r (X, n_1)$ and $\beta \in z_s (X|D,n_2)$ are cycles such that $\alpha \boxtimes \beta$ lies in $z_{\dim X + s -r} (X \times X|D_X, n_1 + n_2)_{\Delta}$. Then, we define the cap product $\alpha \cap_X \beta:= \Delta_{X} ^* (\alpha \boxtimes \beta)$. In  case $D=\emptyset$, we denote it by $\alpha \cup_X \beta$, call it the cup product.
\end{defn}

Lemma \ref{lem:ML-Mod*} below improves \cite[Lemma 4.7]{KL}.

\begin{lem}\label{lem:ML-Mod*}
Fix integers $r, s \ge 0$. Let $\sW$ be a finite set of closed subsets $W_n\subset X\times \square^n$, for $0 \leq n \leq N$, such that each $W_n$ is the support of a cycle in $z_s(X|D, n)$. Let $f:X\to Y$ be a morphism of quasi-projective $k$-schemes. Then, there is a finite set $\sC$ of irreducible locally closed subsets of $Y$ such that for all $\alpha \in z^r _{\mathcal{C}}(Y, *)$ and all $\beta \in \sW$, the cycle $f^* (\alpha)$ lies in $z^r (X,*)$ and the external product $f^*(\alpha) \boxtimes \beta$ lies in $z_{\dim X + s-r}(X \times X |D_X,  *)_{\Delta}$.
\end{lem}

We remark that unlike \cite[Lemma 4.7]{KL}, with an aid of the containment lemma (Proposition \ref{prop:CL*}), we \emph{no longer} need to assume that $f: X \to Y$ is smooth.

\begin{proof} Let $\mathcal{W}_f$ be the collection of $W_n \cap (X \times F)$, where $W_n \in \mathcal{W}$ and $F \subset \square^n$ is a face. This $\mathcal{W}_f$ is a finite collection over $X$ as in Definition \ref{defn:refined moving}. By Lemma \ref{lem:distinguished complex}, there is a finite collection $\mathcal{W}'$ over $Y$ as in Definition \ref{defn:refined moving} such that $f^*: z_{\mathcal{W}'} ^r (Y, \bullet) \to z_{\mathcal{W}_f} ^r (X,\bullet)$ is well-defined. But, by Lemma \ref{lem:refined moving}, there is a finite collection $\mathcal{C}$ of irreducible locally closed subsets of $Y$ and a set function $e: \mathcal{C} \to \mathbb{Z}_{\geq 0}$ such that $z_{\mathcal{C},e} ^r (Y, \bullet) = z_{\mathcal{W}'} ^r (Y, \bullet)$. Since $z_{\mathcal{C}} ^r (Y, \bullet) \subset z_{\mathcal{C}, e} ^r (Y, \bullet)$, we see $f^*: z_{\mathcal{C}} ^r (Y, \bullet) \to z_{\mathcal{W}_f} ^r (X, \bullet)$ is well-defined. We claim $\mathcal{C}$ satisfies the desired properties.

Observe that for each irreducible cycle $Z' \in z^r _{\mathcal{W}_f} (X, m)$, each $W_n \in \mathcal{W}$, and each face $F  \subset \square^{m+n}$, that $\Delta_{X, m+n} ^{-1} (Z' \boxtimes W_n)$ intersects properly with $X \times F$ is equivalent to that $(Z' \cap (X \times F_1)) \times F_2$ intersects properly with $p_{F_1, F_2} ^{-1} (W_n \cap (X \times F_2))$, where $F= F_1 \times F_2$ for faces $F_1 \subset \square^m$ and $F_2 \subset  \square^n$, and $p_{F_1, F_2} : X \times F_1 \times F_2 \to X \times F_2$ is the projection. Since each $W_n \cap (X \times  F_2)$ is a member of $\mathcal{W}_f$, the cycle $Z'$ does satisfy the above proper intersection condition.

Let $Z \in z_{\mathcal{C}}^r (Y, m)$ be an irreducible cycle. By construction, $f^* (Z) \in z^r _{\mathcal{W}_f}(X, m)$ so that Lemma \ref{lem:Ext-Mod} implies that $f^* (Z) \boxtimes W_n \in {z}_{\dim X + s -r } (X \times X|D_X, *)$. By the above observation, $\Delta_{X, m+n}^{-1} (f^* (Z) \boxtimes W_n)$ intersects properly with $X \times F$ for each face $F \subset \square^{m+n}$, so we have checked Definition \ref{defn:good-int}(i) for $f^* (Z) \boxtimes W_n$. To check Definition \ref{defn:good-int}(ii) for $f^* (Z) \boxtimes W_n$, we only need to show that every irreducible component $V$ of $\Delta_{X, m+n} ^* ( f^* (Z) \boxtimes W_n)$ has modulus $D$ on $X \times \square^{m+n}$. 

Let $\nu_V: \ov{V} ^N \to \ov{V} \hookrightarrow X \times \ov{\square} ^{m+n}$ be the normalization of the closure $\ov{V}$ of $V$ in $X \times \ov{\square}^{m+n}$, composed with the closed immersion. Since $f^* (Z) \boxtimes W_n \in z_{\dim X + s -r} (X \times X|D_X,*)$, it has modulus $D_X$ on $X \times X \times \square^{m+n}$. Via the closed immersion $\Delta_{X, m+n} : X \times \square ^{m+n} \inj X \times X \times \square^{m+n}$, $V$ can be seen as a closed subvariety of $|f^*(Z) \boxtimes W_n|$, and hence Proposition \ref{prop:CL*} implies that $V$ has modulus $D_X$ on $X \times X \times \square^{m+n}$. That is, we have $\nu_V ^* \Delta_{X, m+n}  ^* (D_X \times \ov{\square} ^{m+n}) \leq \nu_V ^* \Delta_{X, m+n} ^* (X \times X \times F_{m+n} ^1)$ on $\ov{V}^N$ . Since $\Delta_{X, m+n} ^* (D_X \times \ov{\square}^{m+n}) = D \times \ov{\square} ^{m+n}$ and $\Delta_{X, m+n} ^* (X \times X \times F_{m+n} ^1) = X \times F_{m+n} ^1$, this is equivalent to $\nu_V ^* (D \times \ov{\square}^{m+n}) \leq \nu_V ^* (X \times F_{m+n} ^1)$, which means $V$ has modulus $D$. Hence, $f^* (Z) \boxtimes W_n$ satisfies Definition \ref{defn:good-int}(ii), finishing the proof.
\end{proof}

\begin{lem}\label{lem:Associative}
Let $s$,  $\sW$ and $f:X\to Y$ and $\mathcal{C}$ be as in Lemma~\ref{lem:ML-Mod*}.
Let $\mathcal{C}'$ be a finite collection of locally closed subsets of $Y$ containing $\sC$. Let $\sT$ be a finite collection of closed subsets of $Y\times \square^n$ of the form $\Supp(T_n)$, for some $T_n \in {z}^r _{\mathcal{C}'} (Y,n)$. Let $g:Y\to Y'$ be a morphism of quasi-projective $k$-schemes. Then, there is a finite set $\mathcal{C}''$ of locally closed subsets of $Y'$ such that for each $W\in  \sW$, $Z\in \sT$ and $V\in {z}^q _{\mathcal{C}''} (Y',*)$, we have
\begin{enumerate}
\item the cycles $g^*(V)$, $(g \circ f)^*(V)$, $g^*(V)\cup_YZ$, $(g \circ f)^*(V)\cap_X(f^*(Z)\cap_XW)$ and $f^*(g^*(V)\cup_YZ)\cap_XW$ are all defined,
\item $(g\circ f)^*(V)\cap_X(f^*(Z)\cap_XW) = f^*(g^*(V)\cup_YZ)\cap_XW$ in $z_*(X|D, *)$.
\end{enumerate}
\end{lem}

\begin{proof}
It suffices to prove the lemma when $\sW = \{W\}$ and $\sT = \{Z\}$ are singleton sets, where $W \in z_{s}(X|D, *)$, and the cycle $Z \in z^r_{\mathcal{C}'} (Y, *)$ is integral. Given such $Z \in \mathcal{T}$, the cycle $f^* (Z) \cap_X W$ is in $z_{s-r} (X|D, *)$ by Lemma \ref{lem:ML-Mod*}. Let $\mathcal{U}$ be the collection of intersections of $\Supp(f^* (Z) \cap _X W)$ with all faces $X \times F$. We may apply Lemma \ref{lem:ML-Mod*} to $\mathcal{U}$ and the morphism $g \circ f: X \to Y'$ to yield a finite collection $\mathcal{C}'' (g\circ f)$ of locally closed subsets of $Y$, for which both the pull-back $(g \circ f)^* (V)$ and the cap product $(g \circ f)^* (V) \cap _X (f^* (Z) \cap _X W)$ are well-defined.

Similarly, applying Lemma \ref{lem:ML-Mod*} to the finite collections $\mathcal{U}'$ of all intersections of $\Supp (Z)$ with the faces $Y \times F$, and $\mathcal{U}''$ of all intersections of $\Supp (f^* (Z))$ with the faces $X \times F$, we obtain the finite collection $\mathcal{C}''(g)$ of locally closed subsets of $Y'$, for which the pull-back $g^* (V)$ and the cup products $g^* (V) \cup_Y Z$ and $(g \circ f)^* (V) \cup _X f^* (Z)$ are well-defined. Note that $f^* (Z)$ is already well-defined by our given choice of $\mathcal{C}$. Let $\mathcal{C}'' := \mathcal{C}'' (g) \cup \mathcal{C}'' (g \circ f)$. Now, by construction, $g^* (V) \cup _Y Z$ is a higher Chow cycle for which $f^* (g^* (V) \cup_Y Z)$ and $f^* (g^* (V) \cup_Y Z) \cap_X W$ are well-defined. The equality of (2) is obvious by observing that both are effective cycles.
\end{proof}

\begin{thm}\label{thm:product}
Let $X$ be a smooth quasi-projective $k$-scheme with an effective Cartier divisor $D$. Then, there is an associative product  
\begin{equation}\label{eqn:Prod-0}
\cap_X: \CH^r(X, n_1) \otimes \CH_s(X|D, n_2)\to \CH_{s-r}(X|D, n_1 + n_2),
\end{equation}
natural with respect to flat pull-back, and satisfying the projection formula 
$f_*( f^*(a)\cap_Xb)=a\cap_Y f_*(b)$ for a proper morphism of smooth quasi-projective $k$-schemes with effective divisors $f: (X, D) \to (Y, E)$. If $f$ is a flat and proper morphism with $D = f^*(E)$, we have in addition the projection formula $f_*(a\cap_X f^*(b))=f_*(a)\cap_Y b.$
\end{thm}

\begin{proof}
%
%
To see the existence of the product $\cap_X$, let $\alpha \in \CH^r (X, n_1)$, $\beta \in \CH_s (X|D, n_2)$ be the given cycle classes. Choose a cycle representative $\sum_{j} m_j \beta_j \in z_s (X|D, n_2)$ for $\beta$, with $m_j \in \mathbb{Z}$ and $\beta_j$ irreducible. 


By Lemma \ref{lem:ML-Mod*}, for each $\beta_j$ there exists a finite set $\mathcal{C}_j$ of irreducible locally closed subsets of $X$ (in the notation of Lemma \ref{lem:ML-Mod*}, we take $f= {\rm Id}_X$) such that for each $\alpha' \in z_{\mathcal{C}_j} ^r (X, n_1)$, the external product $\alpha' \boxtimes \beta_j$ lies in $z_{\dim X + s-r} (X \times X | X \times D, n_1 + n_2)_{\Delta}$. Let $\mathcal{C} = \cup_j \mathcal{C}_j$. Here, we have $z_{\mathcal{C}} ^r (X, n_1) \subset z_{\mathcal{C}_j} ^r (X, n_1)$ for each $j$. By Theorem \ref{thm:ML2}, there always exists a cycle representative $ \alpha' \in z_{\mathcal{C}} ^r (X, n_1)$ of $\alpha$. For each such $\alpha ' \boxtimes \beta_j$, there exists a pull-back $\Delta^* (\alpha' \boxtimes \beta_j) \in z_{s-r} (X, n_1 + n_2)$, regarded as a higher Chow cycle. The modulus condition holds by the containment lemma, Proposition \ref{prop:CL*}. Now, $\Delta^* (\alpha' \boxtimes \beta)$ is given by $\sum_j m_j \Delta^* (\alpha' \boxtimes \beta_j) $ and it defines $\alpha \cap_X \beta \in \CH_{s-r} (X|D, n_1 + n_2)$. Associativity follows directly from Lemma \ref{lem:Associative}.



If $f: (X, D) \to (Y,E)$ is a proper map of smooth quasi-projective schemes with effective divisors, the push-forward of a cycle with modulus is defined by Proposition \ref{prop:PFF}. So, to prove the projection formula, given an effective cycle $W \in z_s(X|D, \bullet)$, it is enough to find, using Theorem~\ref{thm:ML2}, a quasi-isomorphic subcomplex $z^r_{\mathcal{C}'}(Y, \bullet)\subset z^r(Y,\bullet)$ such that $V \cap_Y f_*(W)$ and $f^*(V) \cap_X W$ are both defined for all $V \in z^r_{\mathcal{C}'}(Y,*)$. But this follows from Lemma~\ref{lem:ML-Mod*}. If $f$ is flat and proper with $f^*(E) = D$, the flat pull-back of a cycle with modulus is defined by Proposition~\ref{prop:FPB}. The projection formula follows similarly from \lemref{lem:ML-Mod*}.
\end{proof}

\section{Applications of \thmref{thm:product}}
\label{sect:PBF-PB}
In this section, we apply \thmref{thm:product} to show that higher Chow groups with modulus satisfy the projective bundle formula and the blow-up formula. As another application, we show that higher Chow groups with modulus admit pull-back maps for certain classes of smooth schemes with effective divisors. Let $\sD(\Ab)$ denote the (unbounded) derived category of abelian groups.

\subsection{Pull-backs}\label{sec:mod pull-back}
Let $(S,D)$ be a smooth quasi-projective $k$-scheme with an effective Cartier divisor. For a morphism $p:Z \to S$ in $\Sm_S$, we denote the effective divisor $p^*(D)$ by $D$ in what follows. Let $X, Y \in \Sm_S$, with $p_X: X \to S$ and $p_Y: Y \to S$ the structure morphisms. Let $p_1: X \times_S Y \to X$ and $p_2: X \times_S Y \to Y$ be the natural projections.

\begin{defn}[{cf. \cite[Definition 5.1]{KL}}] 
Suppose $p_Y$ is projective. Given $\alpha\in \CH_{r+\dim_S (X)}(X\times_S Y)$, define $\alpha_*:z_s(Y|D, \bullet)\to z_{s+r}(X|D, \bullet)$ to be the composition $z_s(Y|D, \bullet)\overset{p_2^*}{\to} z_{s+\dim_S (X)}(X\times_SY|D, \bullet) \overset{\alpha \cap - }{\to} z_{s+r}(X\times_SY| D, \bullet)\overset{p_{1*}}{\to} z_{s+r}(X|D, \bullet)$ in $\mathcal{D}(\Ab)$. Here, $p_{1*}$ is defined by Proposition \ref{prop:PFF} since $p_Y$ (and hence $p_1$) is projective. 
\end{defn}

\begin{prop}[{cf. \cite[Proposition 5.2]{KL}}]\label{prop:CorrFunct} 
For $X,Y,Z$ smooth and quasi-projective over $S$, with $Y$ and $Z$ projective over $S$, and for $\alpha\in\CH_{r + \dim_S(X)}(X \times_S Y)$, $\beta\in \CH_{r' + \dim_S(Y)}(Y \times_S Z)$, we have $(\beta\circ \alpha)_*=\beta_*\circ\alpha_*$ as maps in $\mathcal{D}(\Ab)$ from $z_s(Z|D,\bullet)$ to $z_{s+r+r'}(X|D,\bullet)$.
\end{prop}

\begin{proof} 
The proof is standard combining Theorem~\ref{thm:product}, the functoriality of flat pull-back and projective push-forward, associativity, compatibility of projective push-forward and flat pull-back in transverse Cartesian squares, and the projection formula for a smooth projective morphism, as in \S \ref{sect:PPFP} (see \cite[Proposition~5.2]{KL}).
\end{proof}

\begin{thm}\label{thm:Pull-back-map}
Let $f:X \to Y$ be a morphism of smooth and quasi-projective schemes over $S$, with $Y$ projective over $S$. Then, there is a well-defined pull-back map $f^*:\CH^s(Y|D,\bullet)\to \CH^s(X|D,\bullet)$ with $(gf)^*=f^*g^*$ when $g: Y \to Z$ is another morphism with $Z$ smooth projective over $S$. If $f$ is flat, then $f^*$ is equal to the flat pull-back. It satisfies the projection formula $ f_*(a\cap_Xf^*(b))=f_*(a)\cap_Yb$ for $a\in \CH^r(X)$, $b\in \CH_s(Y|D,\bullet)$ if $f$ is proper.
\end{thm}

\begin{proof}For $[{}^t\Gamma_f]\in \CH_{\dim_S(X)}(Y\times_S X)$, define $f^* := [ {}^t\Gamma_f]_*$, where ${}^t \Gamma_f$ is the transpose of the graph of $f$. The functoriality follows from $ [{}^t\Gamma_f]\cdot [{}^t\Gamma_g]=[{}^t\Gamma_{gf}]$ in $\CH^*(Z \times_S Y \times_S X)$ and Proposition~\ref{prop:CorrFunct}. That the new definition of $f^*$ agrees with the old one for flat $f$ follows from the identity $ ({\rm Id}_X,f)_*(f^*_{{\rm old}}(w))=[\Gamma_f]\cap p_2^*(w).$

The operations $a \cap_X (-)$ and $f_* (a) \cap _Y (-)$ can be written as the actions of correspondences, namely $\Delta_{X*} (a)_*$ and $\Delta_{Y*}(f_* (a))_*$, where $\Delta_X: X \to X \times_S X$ and $\Delta_Y: Y \to Y \times_S Y$ are the diagonals. Furthermore, $f_*$ is given by $[ \Gamma_f]_*$. The projection formula follows from Proposition \ref{prop:CorrFunct} and the equality of correspondences $[\Gamma_f]\circ \Delta_{X*}(a)\circ {}^t[\Gamma_f]=\Delta_{Y*}(f_*(a)).$
\end{proof}

\begin{remk}\label{remk:Add-Chow-PB}
Take $S = \A^1 = \Spec(k[t])$ and $D = \{t^{m+1} = 0\}$ for $m \ge 1$. Let $f: X \to Y$ be a morphism of smooth quasi-projective $k$-schemes, with $Y$ projective. Let $X' = X \times \A^1$, $Y' =Y \times \A^1 $ and $f' = f \times {\rm Id}_{\A^1}$. Then, by Theorem \ref{thm:Pull-back-map}, we deduce $f^*: \TH^*(Y, \bullet, m) \to \TH^*(X, \bullet, m)$ of \cite[Theorem~7.1]{KP}.
\end{remk}

\subsection{Homological Chow motives}\label{sec:CHmot}
 The notion of homological Chow motives was envisioned by A. Grothendieck and several papers in the literature cover this topic. See e.g. \cite{Scholl}. We recall the version over a base scheme $S \in \Sm_k$, mainly from \cite[\S 2.1]{KL}. 

When $S$ is irreducible and given two irreducible $X, Y \in \SmProj_S$, let $\Cor_S ^n (X, Y) =\CH_{\dim_S X - n} (X \times_S Y)$. We extend it to any $X, Y \in \SmProj_S$ and $S \in \Sm_k$ naturally by taking the direct sums over irreducible components. By definition, objects of the category $\Cor_S$ are pairs $(X, n)$ with $X \in \SmProj_S$ and $n \in \mathbb{Z}$, and morphisms are $\Hom_{\Cor_S} \left( (X,n), (Y,m) \right):= \Cor_S ^{m-n} (X, Y)$. Given $X, Y, Z \in \SmProj_S$ with $\alpha \in \Cor_S ^* (X, Y)$ and $\beta \in \Cor_S ^* (Y, Z)$, the composition is defined by $\beta \circ \alpha:= p_{XZ*} ^{XYZ} ( p_{YZ} ^{XYZ*} (\beta) \cup p_{XY} ^{XYZ*} (\alpha))$, where $p_{XZ} ^{XYZ}$, etc. are the projections defined in the obvious way. Given $(X, n), (Y, m) \in \Cor_S$, we have $(X, n) \otimes (Y,m):= (X \times_S Y, n+m)$. The unit object $1$ is $(S, 0)$. So, $\Hom_{\Cor_S} (1, (X, -n)) = \CH_n (X)$ and $\CH$ that sends $(X,n)$ to $\CH_{-n} (X)$ defines a functor $\CH: \Cor_S \to (\textbf{Ab})$. The category $\mot (S)$ of homological Chow motives is defined to be the pseudo-abelian hull of $\Cor_S$, i.e. its objects are $(X, n, \alpha)$ with an idempotent $\alpha \in \End_{\Cor_S} (X)$ and its morphisms are $\Hom_{\mot(S)} ((X, n, \alpha), (Y, m, \beta)):= \beta \Hom_{\Cor_S} ((X, n), (Y, m)) \alpha$. For each $r \in \mathbb{Z}$, let
$$m \left< r \right> : \SmProj_S \to \Cor_S \hookrightarrow \mot (S)$$ be the composition of functors, where the first functor sends an object $X$ to $(X,r)$ and a morphism $f: X \to Y$ to the graph $\Gamma_f \subset X \times_S Y$, and the second functor maps $(X, n)$ to $(X, n, {\rm id})$, which is a full tensor embedding. We write $m(X)\left< n \right> := (X, n, {\rm id})$. 

 We can generalize the discussion of \cite[\S 5.1]{KL} proven for additive higher Chow groups to higher Chow groups with modulus, using results in \S \ref{sec:module}, especially Theorem \ref{thm:product} and in \S \ref{sec:mod pull-back}. 
 
 Let $(S, D)$ be a smooth quasi-projective $k$-scheme with an effective Cartier divisor $D$. As in \S \ref{sec:mod pull-back}, for a morphism $p: X \to S$ in $\Sm_S$, we denote the effective divisor $p^* (D)$ by $D$. In the words of Chow motives over $S$, Proposition \ref{prop:CorrFunct} implies:

\begin{thm}[{cf. \cite[Theorem 5.3]{KL}}]\label{thm:homChowS} 
Let ${\rm Gr}\Ab$ be the category of graded abelian groups. For each integer $s \geq 1$, the assignment $(X, n) \mapsto \CH_{-n} (X|D, s)$ extends to an additive functor
$$\CH_* (- | D,s): \mot (S) \to {\rm Gr}\Ab$$
where $\CH_* (- | D,s) ( m (X) \left< n \right> , \alpha) = \alpha_* ( \CH_{-n} (X|D, s)) \subset \CH_{-n} (X|D, s).$ 
\end{thm}

\subsection{Projective bundle formula and blow-up formula}\label{sec:mod proj bundle}

Let $(S,D)$ be a smooth quasi-projective $k$-scheme with an effective Cartier divisor. Let $E$ be a vector bundle on $S$ of rank $r+1$. Let $p:  {\P}(E) \to S$ be the associated projective bundle over $S$. Denote $p^*(D)$ by $D$ for simplicity. 

\begin{thm}[{cf. \cite[Theorem 5.6]{KL}}]\label{thm:PBF-Main}
Let $(S, D)$ and $p: \mathbb{P}(E) \to S$ be as above, and let ${\eta} \in \CH^1(\P(E))$ be the cycle class of the tautological line bundle $\sO(1)$. For any $q , n \ge 0$, the map $\theta: \bigoplus_{i= 0}^r \CH^{q-i}(S|D, n)\to \CH^q(\P(E)|D, n)$ given by $(a_0, \cdots , a_r) \mapsto \sum_{i=0} ^r \eta^i \cap_{\P(E)} p^*(a_i)$ is an isomorphism of $\CH^*(S)$-modules. 
\end{thm} 

\begin{proof}From Theorem \ref{thm:homChowS}, we deduce Theorem \ref{thm:PBF-Main} formally using the decomposition $\sum_{i=0} ^r \alpha_i : m (\mathbb{P}(E)) \simeq \bigoplus_{i=0} ^r m(X) \left< i \right>$ as in \cite[Theorem 5.6]{KL}. Here, we give a direct short proof using Proposition \ref{prop:PFF}, Theorem \ref{thm:product} and elementary arguments.


For injectivity of $\theta$, suppose that $\theta(a_0, \cdots , a_r) = 0$. Applying $p_*$ by Proposition \ref{prop:PFF}, we get $\sum_{i=0} ^r p_*(\eta^i \cap_{\P(E)} p^*(a_i)) =0$. By the projection formula in Theorem \ref{thm:product}, this means $\sum_{i=0} ^r  p_*(\eta^i) \cap_{S} a_i = 0$. From the known computations of the Chow groups of projective bundles, we get $a_r = 0$. Applying the operation $p_* ( \eta \cap (-))$ repeatedly on $\theta (a_0, \cdots, a_r)$, we deduce inductively that all $a_i  = 0$. Thus $\theta$ is injective.

To prove that $\theta$ is surjective, let $p_1, p_2 : \P(E) \times_S \P(E) \to \P(E)$ be the projections to the first and the second factor. Let $\Delta: \P(E) \to \P(E) \times_S \P(E)$ be the diagonal. For the graph cycle $[\Gamma_{\Delta}]$, note that $[\Gamma_{\Delta}]_*= {\rm Id}$ on $\CH^q (\mathbb{P}(E)|D, n)$, i.e., $p_{1*} \circ ([\Gamma_{\Delta}] \cap (-)) \circ p_2 ^* = {\rm Id}$. By the K\"unneth decomposition of the diagonal of a projective bundle, we have $[\Gamma_{\Delta}] = \sum_{i=0} ^r (-1)^i p_1 ^* (\eta^i)  \times_S p_2 ^*( \eta^{r-i})$ in the Chow group. Thus, for $\alpha \in \CH^q (\mathbb{P}(E) |D, n)$, $\alpha  =  p_{1*} \circ ([\Gamma_{\Delta}]\cap (-)) \circ p_2 ^* (\alpha)  =  \sum_{i=0} ^r (-1)^i p_{1*} (( p_1 ^* ( \eta^i)\times p_2^* (\eta^{r-i}) )\cap p_2 ^* (\alpha))$
$ =^{\dagger} \sum_{i=0} ^r (-1)^i  \eta^i \cap (p_{1*} (p_2 ^* (\eta^{r-i} \cap \alpha))) = ^{\ddagger} \sum_{i=0} ^r (-1)^i \eta^i \cap p^* (p_* (\eta^{r-i} \cap \alpha)),$
where $\dagger$ holds by the projection formula and $\ddagger$ by Proposition \ref{prop:pp}. So, letting $a_i = (-1)^i p_* (\eta^{r-i} \cap \alpha)$, we get $\theta (a_0, \cdots, a_r) = \alpha$. That $\theta$ is a homomorphism of $\CH^* (S)$-modules follows immediately from Theorem \ref{thm:product}.
\end{proof}

Let $(S,D)$ be a smooth quasi-projective $k$-scheme with an effective Cartier divisor. Let $i:Z \inj X$ be a closed immersion, with $Z$ and $X$ smooth projective over $S$. Let $\pi: X_Z \to X$ be the blow-up of $X$ along $Z$ and let $E$ be the exceptional divisor. Let $i_E: E \hookrightarrow X_Z$ be the induced closed immersion and $q= \pi|_E: E \to Z$.

\begin{thm}[{cf. \cite[Theorem 5.8]{KL}}]\label{thm:Blowup} 
The following sequences are split exact:
\begin{eqnarray*}
& &0 \to \CH_s (E|D,n) \overset{(q_*, -i_{E*})}{\to} \CH_s (Z|D, n) \oplus \CH_s (X_Z|D, n) \overset{i_*+\pi_*}{\to} \CH_s (X|D, n) \to 0,\\
& & 0 \to \CH^s (X|D, n) \overset{(i^*, \pi^*)}{\to} \CH^s (Z|D, n) \oplus \CH^s (X_Z|D, n) \overset{q^* - i_E ^*}{\to} \CH^s (E|D, n) \to 0.
\end{eqnarray*}
\end{thm}

\begin{proof}
This is a straightforward application of Theorems \ref{thm:product} and \ref{thm:homChowS} and the corresponding blow-up formula for the Chow groups (\cite[Proposition~6.7]{Fulton}). By Theorem \ref{thm:homChowS}, we know the functor $\CH_* (-|D, *)$ from $\SmProj_S$ to ${\rm Gr}\Ab$ extends uniquely to $\mot (S)$ and the theorem is a simple consequence of \cite[Proposition 6.7]{Fulton} and Manin's identity principle.  
\end{proof}

\section{Higher 0-cycles with modulus}\label{sec:multi}
Bloch-Esnault \cite{BE1} and R\"ulling \cite{IR} proved that the additive higher Chow groups of 0-cycles of a field are non-trivial. We study the multivariate analogue in this section, and show that these 0-cycle groups in fact vanish. We prove it in more general circumstances of higher Chow groups with modulus in \S \ref{section:0-cycles}. In \S \ref{section:codim 1}, we study codimension $1$-cycles.

\begin{defn}Let $X \in \Sch_k$. Let $r \geq 1$ be an integer. 
When $(t_1, \cdots, t_r) \in \mathbb{A}^r$ are the coordinates, and $m_1, \cdots, m_r \geq 1$ are integers, let $D_{\un{m}}$ be the divisor on $X \times \mathbb{A}^r$ given by the equation $\{ t_1 ^{m_1} \cdots t_r ^{m_r} = 0 \}$. The groups $\CH^q (X\times \mathbb{A}^r | D_{\un{m}}, n)$ are called \emph{multivariate additive higher Chow groups} of $X$. For simplicity, we often say ``a cycle with modulus $\un{m}$'' for ``a cycle with modulus $D_{\un{m}}$.'' 
\end{defn}

\subsection{$0$-cycles}\label{section:0-cycles}

\subsubsection{In characteristic $0$}
We first suppose $k$ is an algebraically closed field of characteristic $0$ unless said otherwise. We aim to show that $\CH^{r+n} (\A^r|D, n) = 0$, when $r \geq 2$, $n \geq 0$ and $D$ belongs to some class of effective Cartier divisors. See Theorem \ref{thm:0 vanishing}.

Recall that a reduced quasi-projective scheme $X$ of dimension $d \ge 1$ over $k$ is \emph{uniruled}, if there is a reduced scheme $Z$ of dimension $d-1$ and a dominant rational map $Z \times \P^1 \dashrightarrow X$ whose restriction to $\{z\} \times \P^1$ is non-constant for some $z \in Z$. The following Lemma \ref{lem:uniruled} might be well-known to the experts, but the authors were not able to find a reference, so we provide its argument.

\begin{lem}\label{lem:uniruled}
Any integral hypersurface $X \subset \mathbb{P}^n$ of degree $d \leq n$ is uniruled.
\end{lem}

\begin{proof}First, consider the special case when $X$ is smooth. Since $d \leq n$, by the adjunction formula (\cite[II-Proposition 8.20, Example 8.20.1, p.182]{Hartshorne}), the anti-canonical bundle of $X$ is ample, i.e. $X$ is Fano. Because ${\rm char} (k) = 0$, by the theorem of Koll{\'a}r-Miyaoka-Mori \cite{KMM1}, \cite{KMM2}, that $X$ is Fano implies that it is rationally connected. Hence it is uniruled. (See \cite[\S V.2]{Kollar}.) This resolves the smooth case.

We now consider the general case. If $n\leq 2$, then $X$ is rational, so we suppose $n \geq 3$. Let $\H_{d,n}$ be the scheme of hypersurfaces in $\mathbb{P}^n$ of degree $d$. Let $M= \begin{pmatrix} n+d \\ d \end{pmatrix}$, which is the number of monomials of degree $d$ in $(n+1)$-variables. Each hypersurface of degree $d$ in $\mathbb{P}^n$ corresponds to a point in the dual projective space $(\mathbb{P}^M)^*$, by mapping the coefficients of a defining equation to the projective coordinate of the coefficients. Hence $\H_{d,n} \simeq \mathbb{P}^M$.

Let $C \subset \H_{d,n}$ be a smooth curve containing the closed point $s\in \H_{d,n}$ that corresponds to $X$. Such $C$ exists because $\mathbb{H}_{d,n} \simeq \mathbb{P}^M$. Let $\pi:\mathcal{X} \to C$ be the universal family of hypersurfaces of degree $d$ parameterized by $C$. This is a closed subscheme of the incidence variety contained in $\mathbb{P}^n \times \mathbb{H}_{d,n}$. Let $\pi_N : \mathcal{X}^N \to \mathcal{X} \overset{\pi}{\to} C$ be the normalization composed with $\pi$. Since ${\rm char} (k) = 0$, by generic smoothness, there is a dense open subset $U \subset C$ such that $\pi^{-1} (U) \to U$ is smooth. In particular, the map $\pi_N ^{-1} (U) \to \pi^{-1} (U)$ is an isomorphism. 

Let $X'= \pi_N ^{-1} (s) \subset \mathcal{X}^N$ be the inverse image of $X = \pi^{-1} (s)$. The general fiber of $\pi$ is a smooth hypersurface in $\mathbb{P}^n$ of degree $\leq n$, so it is uniruled by the smooth hypersurface case we considered already. As $C$ is smooth, by \cite[Corollary IV.1.5.1, p.184]{Kollar}, every closed fiber of $\pi_N$ is also uniruled. Since $X' \to X$ is finite surjective, $X$ must be uniruled by \cite[Lemma IV.1.2, p.182]{Kollar}. 
\end{proof}

\begin{lem}\label{lem:curve}
Let $n \geq 2$. Let $D \subset \A^n$ be an effective Cartier divisor such that $D_{\rm red}$ is a hypersurface of degree $ d \le n$. Then, for each closed point $x \in \A^n \setminus D$, there exists an integral rational affine curve $C \subset \A^n$ such that $x \in C$ and $C \cap D = \emptyset$.
\end{lem}

\begin{proof}We may suppose $D$ is reduced. Since every effective Cartier divisor on $\mathbb{A}^n$ is principal, we may write $D= V(f)$ for some $f \in k[t_1, \cdots, t_n]$, which has degree $\leq n$. Since $k$ is algebraically closed, we can find a linear automorphism $\phi$ of $\mathbb{A}^n$ with $\phi(t_i) = \lambda_0 + \sum_{j=1} ^n \lambda_j t_j$, $\lambda_0, \lambda_j \in k$, such that $\phi (x) = 0$. Since $\phi$ is a linear automorphism, we have $\deg (\phi (f)) = \deg (f)$. So, we reduce to the case when $x = 0$ and $0 \not \in V(f) = D$. By scaling $f$, we may suppose $f(0) = 1$. Write $f = 1 - g$, where $g(0) = 0$ and $\deg (g) = \deg (f)$. Let $g= \prod_{i=1} ^r g_i ^{m_i}$ be the unique factorization of $g$, where each $g_i$ is irreducible, in particular, $(f, g_i) = k[t_1, \cdots, t_n]$. Since $g(0)= 0$, there exists some $i_0$ such that $g_{i_0} (0) = 0$. Note $g_{i_0} | g$ so that $\deg (g_{i_0}) \leq \deg (g) = \deg (f) \leq n$. 

Let $X:= V(g_{i_0})$ and let $\ov{X} \subset \mathbb{P}^n$ be the Zariski closure of $X$. This $\ov{X}$ is an integral hypersurface of $\mathbb{P}^n$ of degree $\leq n$ and is uniruled by Lemma \ref{lem:uniruled}. So, by \cite[Corollary IV.1.4.4, p.184]{Kollar}, there exists an integral rational curve $\ov{C} \subset \ov{X}$ passing through $x \in \ov{X}$. Let $C= \ov{C} \cap X$. This is an integral rational affine curve closed in $X$ through $x$. Since $X \subset \mathbb{A}^n \setminus D$, this $C$ satisfies the desired property.
\end{proof}

\begin{lem}\label{lem:localization}Let $k$ be any field.
For a regular rational affine curve $C$ and $n \geq 0$, $\CH^{n+1}(C, n) = 0$.
\end{lem}

\begin{proof}
We may assume $C$ is connected. For every such $C$, there is an open inclusion $C \hookrightarrow \mathbb{A}^1$, whose complement $Z$ is a finite set of closed points of $\mathbb{A}^1$. In the localization sequence $\CH^{n+1} (\mathbb{A}^1, n) \to \CH^{n+1} (C, n) \to \CH^{n+1} (Z, n-1)$, we know $\CH^{n+1} (\mathbb{A}^1, n) = 0$ by homotopy invariance and $\CH^{n+1} (Z, n-1) = 0$ by the dimension reason. We conclude that $\CH^{n+1} (C, n) = 0$.
\end{proof}

\begin{thm}\label{thm:0 vanishing}
Suppose $k=\ov{k}$ and ${\rm char} (k) = 0$. Let $D \subset \A^r$ be an effective Cartier divisor, with $\deg (D_{\red}) \leq r $. For $r \ge 2$ and $n \geq 0$, $\CH^{n+r}(\A^r|D, n) = 0$.
\end{thm}

\begin{proof}Note that a $0$-cycle has modulus $D$ if and only if it is disjoint from $D \times \square^n$. Let $z \in (\mathbb{A}^r \setminus D) \times \square^n$ be a closed point. We claim that $[z]= 0 $ in $\CH^{n+r} (\mathbb{A}^r|D, n)$. Let $x= p_1(z)$ and $y= p_2 (z)$, where $p_1, p_2$ are the projections from $\mathbb{A}^r \times \square^n$ to $\mathbb{A}^r$ and $\square^n$, respectively. They are closed points, and $x \not \in D$. By Lemma \ref{lem:curve}, we have a closed immersion $\iota:C \hookrightarrow \mathbb{A}^r$ of an integral rational affine curve through $x$ with $C \cap D = \emptyset$. By Corollary \ref{cor:induced map}, there is the push-forward map $\iota_*: \CH^{n+1} (C, n) \to \CH^{n+r} (\mathbb{A}^r|D, n)$, and for $z':=(x,y) \in C \times \square^n$, $\iota_* ([z']) = [z]$. It is therefore sufficient to show that $\CH^{n+1} (C, n) = 0$ in order to prove the theorem.

To prove it, take the normalization $\pi: C^N \to C$. This $C^N$ is a regular connected rational affine curve. Since $\pi$ is finite surjective and $k$ is algebraically closed, the push-forward $\pi_*: \CH^{n+1} (C^N, n) \to \CH^{n+1} (C,n)$ is surjective (see \cite[Proposition 1.3]{Bl1}). We are done by Lemma \ref{lem:localization}.
\end{proof}

\subsubsection{In characteristic $>0$}We now consider the cases  when $k$ is a finite field $\mathbb{F}_q$ or its algebraic closures $\ov{\mathbb{F}}_q$.  For a scheme $X$, let $\mathcal{K}_i ^M$ denote the Zariski sheaf whose stalks are the Milnor $K$-theory of the local rings of $X$.

\begin{lem}\label{lem:SK-1}
Let $X$ be a smooth curve over a field $k$ and let $n \ge 0$. Then, there is a natural isomorphism $\CH^{n+1}(X,n) \overset{\sim}{\to} H^1_{\rm Zar}(X, \sK^M_{n+1})$.
\end{lem}

\begin{proof}
We may assume $X$ is connected. There are exact sequences: 
\begin{equation}\label{eqn:rev_1_SK-1-1}
z^{n+1}(X, n+1)  \overset{\partial}{\to} z^{n+1}(X, n) \to \CH^{n+1}(X,n) \to 0;
\end{equation}
\begin{equation}\label{eqn:rev_1_SK-1-2}
K^M_{n+1}(k(X)) \overset{\delta}{\to} \bigoplus_{x \in X_0}  K^M_{n}(k(x)) \to H^1_{\rm Zar}(X, \sK^M_{n+1}) \to 0.
\end{equation}
The first sequence is exact by definition and the second is exact by Kato's resolution of the Milnor $K$-theory sheaves on smooth schemes (\cite{Kato}), where $X_0$ is the set of closed points of $X$. In the following, we first define maps $\phi_n, \theta_{n+1}$ and $ \psi_n$ that will relate the first two terms of \eqref{eqn:rev_1_SK-1-1} and \eqref{eqn:rev_1_SK-1-2}.

Let $p_X$ and $p_{n}$ be the projections from $X \times \square^{n}$ to $X$ and $\square^{n}$, respectively. Let $q_i: \square^n \to \square$ for $1 \leq i \leq n$ be the projection to the $i$-th $\square.$

We first define $\phi_n: z^{n+1} (X, n) \to \bigoplus_{x \in X_0} K_n ^M (k(x))$ as follows: when $z \in z^{n+1} (X, n)$ is a point, consider $\phi_n ([z]):= N_z ( \{ z_1, \cdots, z_{n} \})$, where $z_i := q_i \circ p_{n}(z)$ and $N_z: K_n ^M (k(z)) \to K_n ^M (k (p_X (z)))$ is the norm map. We extend $\phi_n$ linearly. (Note that this definition makes sense for any $k$-scheme $X$, not just for smooth curves. We write $\phi_n ^X$ for $\phi_n$ if we need to specify $X$.)

Now we define $\theta_{n+1}: z^{n+1} (X, n+1) \to K_{n+1} ^M (k(X))$ as follows: let $(y_1, \cdots, y_i) \in \square^i$ be the coordinates for $i \geq 1$. Let $C \in z^{n+1} (X, n+1)$ be an irreducible curve. If $p_X(C)$ is a point, we define $\theta_{n+1} ([C]) = 0$. Otherwise, the map $p_X|_C: C \to X$ is generically finite. The composites $q_i \circ p_{n+1}|_C: C \overset{}{\to} \square^{n+1} \overset{}{\to} \square$, $1 \leq i \leq n+1$ yield rational functions $f_1, \cdots, f_{n+1}$ on $C$. Proper intersection of $C$ with the faces of $X \times \square^{n+1}$ means that $f_i \not = 0$ for $1 \leq i \leq n+1$. So, they define a unique element $[f]_C:= \{ f_1, \cdots, f_{n+1} \} \in K_{n+1} ^M (k(C))$. We define $\theta_{n+1} ([C]) = N_C ([f]_C)$, where $N_C: K_{n+1} ^M (k(C)) \to K_{n+1} ^M (k(X))$ is the norm map via the generically finite map $p_X|_C$. We extend $\theta_{n+1}$ linearly. 

We claim that the diagram
\begin{equation}\label{eqn:SK-1-2}
\xymatrix{
z^{n+1}(X, n+1) \ar[d]_{\partial} \ar[r]^{\theta_{n+1}} & K^M_{n+1}(k(X)) \ar[d]_{\delta}   \\
z^{n+1}(X, n) \ar[r]^{  \phi_{n} \ \ } &  \underset{x \in X_0}{\bigoplus} K^M_{n}(k(x)),}
\end{equation}
commutes. Here, if $C$ is an irreducible curve in $z^{n+1} (X, n+1)$ such that $p_X (C)$ is a closed point, say $\{x \} \in X$, then one can regard $C$ as $\{ x \} \times C'$ for a curve $C' \subset \square^{n+1}$. Then $\phi_n \circ \partial ([C]) = \phi_n ^X \circ \partial ([C]) = \phi_n ^{\{x \}} \circ \partial ^{\{x \}} ([C'])$, where $\partial ^{\{x \}}: z^{n} (k(x), n+1) \to z^n (k(x), n)$ is the boundary map for $\{x \}$, and the latter one $\phi_n ^{\{x \}} \circ \partial ^{\{x \}} ([C'])$ is $0$ because the map $\phi_n ^{\{x \}} : z^n (k(x), n) \to K_n ^M (k(x))$ sends the boundary $\partial^{\{x\}} ( z^n (k(x), n+1))$ to $0$, inducing the homomorphism $\phi_n ^{\{x \}}: \CH^n (k(x), n) \to K_n ^M (k(x))$ as in the Nesterenko-Suslin--Totaro isomorphism \cite{Totaro}. In particular, $\phi_n \circ \partial ([C]) = \delta \circ \theta_{n+1} ([C]) = 0$. In case $C \to X$ is dominant, we have $\phi_n \circ \partial ([C]) = \delta \circ \theta_{n+1} ([C])$ by the definition of tame symbols in Milnor $K$-theory. See \cite{BT} for details. Hence the diagram \eqref{eqn:SK-1-2} induces a homomorphism $\bar{\phi}_n: {\rm coker} \  \partial  \to {\rm coker} \  \delta$. Note that ${\rm coker} \ \partial  = \CH^{n+1} (X, n) = z^{n+1} (X, n)/ \partial ( z^{n+1} (X, n+1))$ and ${\rm coker} \ \delta = H^1 _{\rm Zar} (X, \mathcal{K}_{n+1} ^M)$ by \eqref{eqn:rev_1_SK-1-1} and \eqref{eqn:rev_1_SK-1-2}.

We now define $\psi_n: \bigoplus_{x \in X_0} K_n ^M (k(x)) \to z^{n+1} (X, n)/ \partial (z^{n+1} (X, n+1)) $ as follows: when $x \in X_0$ and $[f] = \{ f_1, \cdots, f_n \} \in K_n ^M (k(x))$ with $f_i \in k(x)^{\times} $, we let $\widetilde{\psi}_n ([f])$ be the graph of the morphism $(f_1, \cdots, f_n)$, which is a closed point in $X \times \square^n$ if $f_i \not = 1$ for all $i$, or $\emptyset$ if $f_i = 1$ for some $i$. This does not intersect any face of $X \times \square^n$ because none of $f_i$ is $0$ (nor $\infty$, obviously), so it defines an element in $z^{n+1} (X, n)$. Its image in $z^{n+1} (X, n) / \partial (z^{n+1} (X, n+1))$ will be called $\psi_n ([f])$. To see that this is well-defined, it reduces to show that for $f_1, f_3, \cdots, f_n \in k(x)^{\times}$, with $f_1 \not = 1$, the graph of $(f_1, 1-f_1, f_3, \cdots, f_n)$ vanishes in the quotient $z^{n+1} (X, n) / \partial (z^{n+1} (X, n+1))$. For this, we use the curve $\gamma: t \mapsto \{x \} \times  \left( t, 1-t, \frac{ f_1 -1}{1-t}, f_3, \cdots, f_n \right) \subset X \times \square^{n+1},$ where the first three coordinates of $\square^{n+1}$ are exactly those of the rational curve of B. Totaro in \cite[p.182, line 28]{Totaro} used to kill the Steinberg relation. Its only intersection with any codimension $1$ face is the closed point $\{x \} \times (f_1, 1-f_1, f_3, \cdots, f_n)$. Hence $\psi_n ( \{ f_1, 1-f_1, f_3, \cdots, f_n \}) = 0$. This proves the well-definedness of $\psi_n$ as a set map. To show that it is a group homomorphism, it reduces to check that when $n=1$, we have $\psi_1 (f) + \psi_1 (1/f) = 0$ for $f \in k(x) ^{\times} $ and $\psi_1 (f) + \psi_1 (g) = \psi (fg)$ for $f, g  \in k(x) ^{\times} $. This also can be done by taking the curve $\gamma' : t \mapsto \{ x \} \times \left( t, \frac{ ft-fg}{t - fg}\right) \subset X \times \square^2,$ where the two coordinates of $\square^2$ are exactly the the rational curve of B. Totaro in \cite[p.182, line 8]{Totaro}. The boundary of $\gamma'$ gives the first relation when $fg=1$ and the second relation in general. This shows $\psi_1$ as well as $\psi_n$ is a group homomorphism. Extend $\psi_n$ to the direct sum over $X_0$.

We now show that $\psi_n \circ \delta = 0$. Let $[f] \in K_{n+1} ^M (k(X))$. We may assume that this is given by $(n+1)$ distinct rational functions in $k(X) ^{\times} \setminus \{ 1 \}$ for otherwise $[f]=0$ so that there is nothing to prove. Note that $k(X)$ is the fraction field of the dvr $A= \mathcal{O}_{X,x}$ for any chosen $x \in X_0$. Fix one $x \in X_0$. Then we can write $[f] = \{ f_1, \cdots, f_n, u \pi^r \}$, for some $f_i \in A^{\times}$ for $1 \leq i \leq n$, $u \in A^{\times}$, $\pi$ is a uniformizing parameter of $A$, and $r \in \mathbb{Z}$. In this case, by the construction we know that $\delta_x ([f]) = r \{ \bar{f}_1, \cdots, \bar{f}_n \} \in K_n ^M (k(x))$, where $\bar{f}_i \in k(x)$ is the residue class of $f_i$. See \cite[\S 4]{BT}. We repeat it for each $x \in X_0$. To show that $\psi_n \circ \delta ([f]) = 0$ in $z^{n+1} (X, n)/ \partial (z^{n+1} (X, n+1))$, we need to construct a $1$-cycle in $z^{n+1} (X, n+1)$, whose boundary is equal to $\widetilde{\psi}_n \circ \delta ([f])$, where $\widetilde{\psi}_n$ is defined at the beginning of the paragraph that defines $\psi_n$. To do so, construct $\xi_{n+1} (f_1, \cdots, f_n, u\pi^r) \subset X \times \square^{n+1}$ to be the intersection of the graphs of $(f_1, \cdots, f_n, u \pi^r )$ in $X \times \ov{\square}^{n+1}$ with $X \times \square^{n+1}$. Since the rational functions are all distinct, the intersection defines a curve and this curve intersects all faces properly. By construction, we have $\partial (\xi_{n+1} (f_1, \cdots, f_n, u \pi^r)) = \widetilde{\psi}_n \circ \delta ([f])$. Thus we have the induced map $\bar{\psi}_n : {\rm coker} \ \delta \to {\rm coker} \ \partial$.
That $\bar{\phi}_n \circ \bar{\psi}_n$ and $\bar{\psi}_n \circ \bar{\phi}_n$ are the identity maps is straightforward on the generators by definition.
\end{proof}

\begin{lem}\label{lem:SK-1-0}
Let $k = \mathbb{F}_q$ or $\ov{\mathbb{F}}_q$. Let $X$ be a smooth curve over $k$. Then, $\CH^{n+1}(X, n) = 0$ for $n \ge 2$. If $X$ is affine, then $\CH^2(X,1)=0$ as well.
\end{lem}

\begin{proof}
We may assume $X$ is connected. When $n \geq 2$, by Lemma \ref{lem:SK-1}, it suffices to check that $K_n ^M (F)= 0$, when $F = k(x)$ for $x \in X_0$, which is either finite or the algebraic closure of a finite field. This is a result of Steinberg (see \cite[Example~1.3]{Milnor}) when $F$ is finite. When $F$ is the algebraic closure of a finite field, a direct limit argument shows that $K_n ^M (F) = 0$. 

Now suppose $n=1$ and $X= \Spec (A)$ is a smooth affine curve over $k$. In this case, by \cite[Lemma 2.3]{KrSr}, there are isomorphisms $H^1_{\rm Zar}(X, \sK^M_2) \simeq H^1_{\rm Zar}(X, \sK_2) \simeq SK_1(X)$, where $SK_1(X) = \ker (K_1(X) \to \sO^{\times}(X))$. When $k= \mathbb{F}_q$, $SK_1 (X) = 0$ by \cite[Corollary 4.3]{BMS}. When $k= \ov{\mathbb{F}}_q$, there is a finite subfield $k' \subset k$ and a smooth affine $k$-algebra $A'$ of dimension $1$ such that $A \simeq A'\otimes_{k'} k$. This gives $SK_1 (A) = SL(A) / E(A) = \varinjlim_{\ell}  SL(A'_{\ell})/ E(A'_{\ell}) = 0$, where we take the direct limit over all fields $\ell$ such that $k' \subset \ell \subset k, |\ell|< \infty$. Now, $\CH^2 (X,1) = 0$ by Lemma \ref{lem:SK-1}.
\end{proof}

\begin{prop}\label{prop:SK-1-vanish}
Let $k=\mathbb{F}_q$ or $\ov{\mathbb{F}}_q$. Let $X$ be an irreducible curve over $k$. Then, $\CH^{n+1}(X, n) = 0$ for $n \ge 2$. If $X$ is affine, then $\CH^2(X,1)=0$ as well.
\end{prop}

\begin{proof}
We may assume that $X$ is integral. By Lemma \ref{lem:SK-1-0}, the proposition holds when $X$ is smooth. So, suppose $X$ is not smooth. Let $U\subset X$ be the smooth locus of $X$ and let $S= (X\setminus U)_{\rm red}$. Let $f: X' \to X$ be the normalization and let $U':= f^{-1} (U)$. Let $S':= (X' \setminus U')_{\rm red}$. From the morphism of localization exact triangles
\[
\xymatrix{
z^n(S', \bullet) \ar[r] \ar[d] & z^{n+1}(X', \bullet) \ar[r] \ar[d] & z^{n+1}(U', \bullet) \ar[r] \ar[d]^{\simeq} & z^n(S', \bullet)[1] \ar[d] \\
z^n(S, \bullet) \ar[r] & z^{n+1}(X, \bullet) \ar[r] & z^{n+1}(U, \bullet) \ar[r] & z^n(S, \bullet)[1],}  
\]
we obtain the Mayer-Vietoris exact sequence 
\begin{equation}\label{eqn:pushMV}
\CH^{n}(S', n) \overset{(\iota', f_*)}{\to} \CH^{n+1}(X', n) \oplus \CH^{n}(S, n) \overset{f_*+ \iota}{\to}  \CH^{n+1}(X, n) \to \CH^{n}(S', n-1) 
\end{equation}
associated to the push-forward maps, where $\iota$ and $\iota'$ are maps induced by the closed immersions $S \hookrightarrow X$ and $S' \hookrightarrow X'$, respectively. Note that $\dim \ S' = 0$ so that $\CH^{n}(S', n-1) = 0$ by the dimension reason. On the other hand, $S$ is a finite union of reduced closed points so that $\CH^n (S, n) = \bigoplus_{x \in S} \CH^n (k(x), n) \simeq^{\dagger} \bigoplus_{x \in S} K_n ^M (k(x), n) = K_n ^M (\mathcal{O}_S)$, where $\mathcal{O}_S= \prod_{x \in S} k(x)$ and $\dagger$ is the Nesterenko-Suslin--Totaro isomorphism $K^M_n(F) \simeq \CH^n(F, n)$ for any field $F$ (see \cite{Totaro}). Similarly, for $\mathcal{O}_{S'}:= \prod_{x \in S'} k(x)$ we have $\CH^n (S', n) \simeq K_n ^M (\mathcal{O}_{S'})$. The pushforward $\CH^n (S', n) \to \CH^n (S, n)$ corresponds to the norm map $N: K_n ^M (\mathcal{O}_{S'}) \to K_n ^M (\mathcal{O}_S)$. Hence \eqref{eqn:pushMV} yields an exact sequence $K^M_{n}(\sO_{S'}) \overset{N}{\to} K^M_n(\sO_{S}) \to \frac{\CH^{n+1}(X, n)}{f_* \CH^{n+1}(X', n)} \to 0.$ Lemma \ref{lem:SK-1-0} now reduces the problem to show that $N$ is surjective. This is clear if $k= \ov{\mathbb{F}}_q$. If $k= \mathbb{F}_q$ and $n \ge 2$, then $K^M_n(\sO_{S}) = 0$ by Steinberg's theorem (see \cite[Example~1.3]{Milnor}). If $n =1$, then surjectivity of $N$ follows from the following elementary fact: for a finite extension $\mathbb{F}_q \hookrightarrow \ell$, the norm $N: \ell ^{\times} \to \mathbb{F}_q^{\times}$ is surjective.
\end{proof}

\begin{thm}\label{thm:0 vanishing-1}
Let $k= \mathbb{F}_q$ or $\ov{\mathbb{F}}_q$. Let $(X,D)$ be an affine $k$-variety of dimension $r$ with an effective Cartier divisor. Then, $\CH^{n+r} (X|D, n) = 0$ for $r \geq 2$ and $n \geq 1$. 

\end{thm}

\begin{proof}
We may assume $X= \Spec (A)$ is connected. Let $I \subset A$ be the ideal of $D$. Let $z \in z^{n+r} (X|D, n)$ be a closed point. Let $x= p_X (z)$, where $p_X: X \times \square^n \to X$ is the projection. Let $\mathfrak{m} \subset A$ be the maximal ideal of $x$. Since $x \not \in D$, we have $\mathfrak{m} + I = A$. Choose $f \in I$ and $g \in \mathfrak{m}$ such that $(f,g) = A$. Let $E$ be an irreducible component of $V(g)$ that contains $x$, such that $E \cap D = \emptyset$. Considering the chains of prime ideals in $\mathfrak{m}$ modulo $(g)$ and using that ${\rm ht} (\mathfrak{m}) \geq 2$ (since $r \geq 2$), we see there is an integral curve $C \subset E$ such that $x \in C$. In particular, $C \cap D \subset E \cap D = \emptyset$. Let $\iota: C \hookrightarrow X$ be the closed immersion. By Corollary \ref{cor:induced map}, we have $\iota_*: \CH^{n+1} (C, n) \to \CH^{n+r} (X|D, n)$ such that $[z] \in {\rm im} (\iota_*)$. But, $\CH^{n+1} (C, n)= 0$ by Proposition \ref{prop:SK-1-vanish}. So ${\rm im} (\iota_*)$ is trivial and hence $[z] = 0$. We conclude that $\CH^{n+r} (X|D, n) = 0$.
\end{proof}

\subsubsection{$0$-cycles on multivariate additive higher Chow groups} 
\begin{thm}\label{thm:0 vanishing-additive}
Let $k$ be any field. Let $n \ge 0,  r \ge 2$ and $\un{m} = (m_1, \cdots , m_r)$, with $m_i \geq 1$. Then, $\CH^{n+r}(\A^r | D_{\un{m}}, n) = 0$.
\end{thm}

\begin{proof}Let $z \in \mathbb{A}^r \times \square^n$ be a closed point with modulus $D_{\un{m}}$. Let $\ell = k (z)$, which is finite over $k$. Then, there is an $\ell$-rational closed point $w \in \mathbb{A}_{\ell} ^r \times_{\ell} \square_{\ell} ^n$ such that $\pi_* ([w]) = [z]$, where $\pi: (\mathbb{A}_{\ell} ^r , \pi ^* (D_{\un{m}}) ) \to ( \mathbb{A}_k ^n, D)$ is the base change map and $\pi_*$ is as in Proposition \ref{prop:PFF}. Here $\pi^* (D_{\un{m}})= \{ t_1 ^{m_1} \cdots t_r ^{m_r} = 0 \}$ on $\mathbb{A}_{\ell} ^r$ as well. So, we are reduced to showing that $[w] = 0$ in $\CH^{n+r} (\mathbb{A}^r _{\ell} | D_{\un{m}} , n)$. In other words, we may assume $z$ is $k$-rational. In particular, $x= \pi_{\mathbb{A}^r} (z)$ is $k$-rational for the projection $p_{\mathbb{A}^r} : \mathbb{A}^r \times \square^n \to \mathbb{A}^r$.

The rest of the proof is now a copycat of the argument of Theorem \ref{thm:0 vanishing}, if we can show that given any $k$-rational point $x \in \mathbb{A}^r \setminus D_{\un{m}}$, there is an integral rational affine curve $C \hookrightarrow \mathbb{A}^r$ such that $x \in C$ and $C \cap D_{\un{m}} = \emptyset$. If $x = (c_1, \cdots, c_r)$ with $c_i \in k^{\times}$, let $C_1 \subset \mathbb{A}^2$ be the curve given by the polynomial $f= t_1 t_2 - c_1 c_2 \in k[t_1, t_2]$ and set $C = C_1 \times (c_3, \cdots, c_r)$. This $C$ satisfies the desired properties. So, we are done by Lemma \ref{lem:localization}.
\end{proof}

\begin{thm}\label{thm:0 vanishing-2}Let $k$ be any field.
Let $n \geq 0, r \geq 2$ and $\un{m}= (m_1, m_2)$, with $m_i \geq 1$. Let $X$ be an equidimensional $k$-scheme of dimension $r-2$ with an effective Cartier divisor $D$.  Then, $\CH^{n+r} (X \times \mathbb{A}^2| D \times \mathbb{A}^2 + X \times D_{\un{m}}, n) = 0$.
\end{thm}

Its proof is almost identical to that of Theorem \ref{thm:0 vanishing-additive}: after reducing to the case of $k$-rational points, if $z= (z', c_1, c_2) \in (X \times \mathbb{A}^2)(k)$ is away from $D \times \mathbb{A}^2 + X \times D_{\un{m}}$, where $z' \in (X\setminus D)(k)$, then we have $C = z' \times C_1 \ni z$, where $C_1 = \{ t_1 t_2 = c_1 c_2 \} \subset \mathbb{A}^2$.

\subsection{Codimension $1$ cycles}\label{section:codim 1}
Let $r \geq 2$, $n \geq 0$, and $\un{m} = (1, \cdots, 1)$. We now consider $\CH^1 (\mathbb{A}^r |D_{\un{m}}, n)$. For simplicity, we identify $(\square, \{ \infty, 0 \})$ with $\square_{\psi}:=(\mathbb{A}^1, \{ 0, 1 \})$ via the automorphism $\psi: \mathbb{P}^1 \to \mathbb{P}^1$, $y \mapsto 1/ (1-y)$. We  use $\partial = \sum_{i=1} ^n (-1)^i (\partial_i ^0 - \partial_i ^{1})$ as the boundary operator, where $\partial_i ^\epsilon$ is given by $\{ y_i= \epsilon\}$. 

\begin{lem}\label{lem:divisor present}Let $Z \in z^1  (\mathbb{A}^r|D_{\un{m}}, n)$ be an irreducible admissible cycle. Then, $Z= V(f)$, where  $f = 1 - t_1 \cdots t_r g $ for some $g \in k[t_1, \cdots, t_r][y_1, \cdots, y_n]$. 
\end{lem}

Note that we don't claim that every cycle of the form $V(f)$, with $f$ as above, is admissible, unless $n=0$ for which the above is sufficient.

\begin{proof}Since $k[t_1, \cdots, t_r][y_1, \cdots, y_n]$ is a UFD, hence normal, by Krull's principal ideal theorem (e.g. \cite[I-Theorem 1.11A]{Hartshorne}) there exists a polynomial $f$ in the ring with $Z = V(f)$. However, the modulus condition of $Z$ requires that $Z \cap V (t_i ) = \emptyset$ for $1 \leq i \leq r$. For $i=1$, it means that, if we write $f= p_0 + p_1 t_1 + \cdots + p_N t_1 ^N$, where $p_j \in k[ t_2, \cdots, t_r][y_1, \cdots, y_n]$, then $(f, t_1) = (1)$. Hence $p_0$ must be a nonzero constant. By scaling, we may assume $p_0 = 1$. Thus, $f-1$ is divisible by $t_1$. By repeating this argument for all $1 \leq i \leq r$, we see that $f-1$ is divisible by $t_1 \cdots t_r$.
\end{proof}

\begin{thm}\label{thm:codim 1 n=0}
Let $\un{m}= (1, \cdots, 1)$. Then, $\CH^1 (\mathbb{A}^r| D_{\un{m}}, 0) = 0$. 
\end{thm}

\begin{proof}By Lemma \ref{lem:divisor present}, every irreducible $Z \in z^1 (\mathbb{A}^r|D_{\un{m}}, 0)$ is given by an equation of the form $f(t_1, \cdots, t_r) = 1- t_1 \cdots t_r g$ for some $g \in k[t_1, \cdots, t_r]$. Consider $W \subset \mathbb{A}^r \times \square_{\psi}$ given by the polynomial $h(t_1, \cdots, t_r, y_1) := 1 - t_1 \cdots t_r g y_1$. Writing $y_1 = s_1/s_0$, where $[s_0,s_1] \in \ov{\square}_{\psi}= \mathbb{P}^1$ are the homogeneous coordinates, the Zariski closure $\ov{W}$ in $\mathbb{A}^r \times \ov{\square}_{\psi}$ is given by $s_0 = t_1 \cdots t_r g s_1$. Hence, on the normalization of $\ov{W}$, we have $\{ t_1 \cdots t_r = 0 \} \leq \{ s_0 = 0 \} = \{ y_1 = \infty \}$. So, $W$ has modulus $\un{m}$. On the other hand, $h (t_1, \cdots, t_r, 0) = 1$ and $h ( t_1, \cdots, t_r, 1) = f(t_1, \cdots, t_r)$ so that $\partial _1 ^0 W = 0$ and $\partial _1 ^1 W = Z$.  In particular, the intersections of $W$ with faces is proper and $\partial (W) = Z$. Hence $[Z] = 0$ in $\CH^1 (\mathbb{A}^r|D_{\un{m}}, 0)$. 
\end{proof}

\begin{lem}\label{lem:y degree bound}Let $\un{m} = (1, \cdots, 1)$, and $n \geq 1$. Let $W \in z^1 (\mathbb{A}^r|D_{\un{m}}, n)$ be an irreducible admissible cycle. Suppose it is given by a polynomial $h \in k[t_1, \cdots, t_r,][y_1,  \cdots y_n]$. Then, for each $1 \leq i \leq n$, $\deg _{y_i} (h) \leq 1$. 
\end{lem}

\begin{proof}Write $y_i = s_{i,1}/ s_{i,0}$, where $[s_{i,0},s_{i,1}] \in \mathbb{P}^1$ are the homogeneous coordinates of the $i$-th $\ov{\square}$ in $\ov{\square}^n_{\psi}$. After scaling, we may assume $h(t_1, \cdots, t_r, y_1, \cdots, y_n) = 1 - t_1 \cdots t_r g_{\un{m}} + \sum_{\un{i}\not = \un{m}} t^{\un{i}} g_{\un{i}}$, where $g_{\un{m}}, g_{\un{i}} \in k[y_1, \cdots, y_n]$. Let $d_i = \deg_{y_i} (h)$. In terms of the homogeneous coordinates of the space $\mathbb{A}^r \times \ov{\square}^n_{\psi}$, the equation $0 = h$ of the closure $\ov{W}$ can be expressed as $s_{1,0}^{d_1} \cdots s_{n,0} ^{d_n} = t_1 \cdots t_r  \tilde{g},$ for some polynomial $\tilde{g} \in k[ \{ t_\ell,  s_{i, j} \} | \ 1 \leq i \leq n, j= 0, 1, 1 \leq \ell \leq r ]$. Thus, on the normalization of $\ov{W}$, we have $[ \{ t_1 \cdots t_r = 0 \}] \leq [\{ s_{1,0} ^{d_1} \cdots s_{n,0} ^{d_n} \}]$, which is equivalent to $\sum_{i=1}  ^r \{ t_i = 0 \} \leq \sum_{i=1} ^n d_i \{ s_{i,0} = 0 \} = \sum_{i=1} ^n d_i \{ y_i = \infty \}.$ But, if any one of $d_{i} > 1$, it violates the modulus condition. Thus, we have $d_i \leq 1$ for all $i$.
\end{proof}

For $\un{m}= (1, \cdots, 1)$, consider the homomorphism $ \rho: z^1 (\mathbb{A}^r |D_{\un{m}}, 1) \to k$ given as follows. Let $Z \in z^1 (\mathbb{A}^r | D_{\un{m}}, 1)$ be irreducible. We have $Z= V(f)$ for some $f \in k[t_1, \cdots, t_r][y_1, \cdots y_n]$ as in Lemma \ref{lem:divisor present}. Define $\rho$ by $\rho (Z):= {\rm res}_{y_1 = \infty} Ev_{(t_1 = \cdots = t_r = 0)} \left( \frac{ f(0, \cdots, 0, y_1) - f}{ f(0, \cdots, 0, y_1) t_1 \cdots t_r} \right).$ We want to show that for the boundary map $\partial: z^1 (\mathbb{A}^r|D_{\un{m}}, 2) \to z^1 (\mathbb{A}^r|D_{\un{m}}, 1)$, the composition $\rho \circ \partial = 0$. By Lemma \ref{lem:divisor present}, each irreducible cycle $W= V(h) \in z^1 (\mathbb{A}^r|D_{\un{m}}, 2)$ has the form $ h(t_1, \cdots, t_r, y_1, y_2) = 1 - t_1 \cdots t_r g_{\un{m}} + \sum_{\un{i}\not = \un{m}} t^{\un{i}} g_{\un{i}},$ where $\un{i} = (i_1, \cdots, i_r)$ with all $i_j \geq 1$, $t^{\un{i}} := t_1 ^{i_1} \cdots t_r ^{i_r}$ and $g_{\un{m}}, g_{\un{i}} \in k [y_1, y_2].$ Let $\hslash = \sum_{\un{i} \not = \un{m}} t^{\un{i}} g_{\un{i}}$, which is the higher order terms. For each $i=1, 2$ and $\epsilon = 0, 1$, we have $\partial _i ^{\epsilon} W = V( h|_{y_i = \epsilon})$. But, we have $Ev_{(t_1 = \cdots = t_r = 0)} \left( \frac{ 1 -  h(y_i = \epsilon)}{ t_1 \cdots t_r} \right) = Ev_{(t_1 = \cdots = t_r = 0)} \left( \frac{  t_1 \cdots t_r g_{\un{m}}|_{y_i = \epsilon} - \hslash |_{y_i = \epsilon}}{ t_1 \cdots y_r} \right)$ $ = Ev_{(t_1 = \cdots = t_r = 0)} ( g_{\un{m}}|_{y_i = \epsilon})$ so that, it is enough to consider the case when $h$ is replaced by $1- t_1 \cdots t_r g_{\un{m}}$, i.e., without the higher order terms $\hslash$. Now, that $\rho \circ \partial = 0$ can be checked readily in the following:

\begin{lem}\label{lem:rho recip}
Let $\un{m} = (1, \cdots, 1)$. Let $h = 1 - t_1 \cdots t_r g_{\un{m}}(y_1, y_2)$ and $W= V(h) \in z^1 (\mathbb{A}^r|D_{\un{m}}, 2)$. Then, $\rho \partial (W) = 0$. It induces a homomorphism $\rho: \CH^1 (\mathbb{A}^r |D_{\un{m}} , 1) \to k$.
\end{lem}

\begin{proof}By Lemma \ref{lem:y degree bound}, we must have $\deg_{y_i} g_{\un{m}} \leq 1$ for $i=1, 2$. Thus, the most general form of $g_{\un{m}}(y_1, y_2)$ is $g_{\un{m}}(y_1, y_2) = a y_1 y_2 + b y_1 + c y_2 + d,$ for some $a, b, c, d \in k$. By the identical method used in Lemma \ref{lem:y degree bound}, this cycle is admissible. 

Let's compute $\rho (\partial_i ^j W)$ for each $i = 1, 2$ and $j= 0, 1$. To compute the faces, note that for $\partial_1 ^0 (W)$, we intersect $(y_1 = 0)$ with $W$. This gives $h|_{y_1 = 0} = 1- t_1 \cdots t_r (cy_2 + d)$, and rename the coordinate $y_2$ by $y_1$. This gives $\partial_1 ^0 (W) = [V( 1- t_1 \cdots t_r ( cy_1 + d))]$. By the same method, we obtain, $\partial_1 ^1 (W) = [V( 1- t_1 \cdots t_r (  (a+ c) y_1 + b+d))]$, $\partial _2 ^0 (W) = [V( 1- t_1 \cdots t_r ( b y_1 + d))]$, and $\partial_2 ^1 (W) = [V (1 - t_1 \cdots t_r (  (a+b) y_1 + c+d))]$. Thus, $\rho (\partial W) = c - (a+c) - b + (a+b) = 0.$ The second part follows from the above discussion.
\end{proof}

\begin{thm}\label{thm:codim 1 n=1}Let $r \geq 2$ and let $\un{m} = (1,\cdots, 1)$. Then, $\rho: \CH^1 (\mathbb{A}^r|D_{\un{m}}, 1) \to k$ is surjective. In particular, $\CH^1 (\mathbb{A}^r |D_{\un{m}}, 1) \not = 0$.
\end{thm}

\begin{proof}By Lemma \ref{lem:rho recip}, it is enough to show that $\rho$ is surjective. For any $a \in k$, consider the cycle $Z_a = V (1 - t_1 \cdots t_r ay_1)$. One checks $Z_a$ has modulus $\un{m}$, exactly as we did in the middle of the proof of Theorem \ref{thm:codim 1 n=0}. Here $\partial_1 ^0 (Z_a) = [V ( 1) ] = [\emptyset] = 0$, while $\partial_1 ^1 (Z_a) = [V (1- t_1 \cdots t_r)]$ is a degenerate cycle, which is $0$. So $Z_a$ intersects all faces properly and $\partial (Z_a) = 0$, i.e., it represents a class in $\CH^1 (\mathbb{A}^r|D_{\un{m}}, 1)$. 
As $\rho (Z_a) = a$ by definition, we conclude that $\rho$ is surjective.
\end{proof}

We do not know whether $\rho$ is injective, nor what $\CH^1 (\mathbb{A}^r |D_{\un{m}}, n)$ is when $n \geq 1$. 

\vspace*{.5cm}

\noindent\emph{Acknowledgments} JP thanks Moritz Kerz, Kay R\"ulling, and Shuji Saito for some inspiring conversations. 
JP thanks Juya for helps at home and Damy for being born during the work.  The authors are grateful fo the editor and the referee of Mathematical Research Letters for their numerous helps in improving the quality of this article.

During this research, he was partially supported by the National Research Foundation of Korea (NRF) grant (No. 2013042157) and Korea Institute for Advanced Study (KIAS) grant, both funded by the Korean government (MSIP), and TJ Park Junior Faculty Fellowship funded by POSCO TJ Park Foundation. AK thanks the mathematics department of KAIST for invitation in August 2014, where part of this work was done.



\begin{thebibliography}{99}


\bibitem{BMS} H. Bass, J. Milnor, J.-P. Serre, {\sl Solution of the congruence subgroup problem\/}, Publ. Math. Inst. Hautes {\'E}tudes Sci. {\bf 33}, (1967), 59--137. 

\bibitem{BT} H. Bass, J. Tate, {\sl The Milnor ring of a global field\/}, Algebraic $K$-theory II, (Proc. Conf. Seattle),  Lect. Notes Math., {\bf 342}, (1972), 349--446. 

\bibitem{BS} F. Binda, S. Saito, {\sl Relative cycles with moduli and regulator maps\/}, preprint arXiv:1412.0385, (2014).

\bibitem{Bl1} S. Bloch, {\sl Algebraic cycles and higher $K$-theory\/}, Adv. Math., \textbf{61}, (1986), 267--304.

\bibitem{Bl2} S. Bloch, {\sl The moving lemma for higher Chow groups\/}, J. Algebraic Geom., \textbf{3}, (1994), 537--568.

\bibitem{BE1} S. Bloch, H. Esnault, {\sl The additive dilogarithm\/}, Doc. Math., \textbf{Extra Vol.} Kazuya Kato's fiftieth birthday, (2003), 131--155.

\bibitem{Fulton} W. Fulton, {\sl Intersection theory\/}, 2nd Edition, Ergeb. Math. Grenzgebiete, {\bf 2}, Springer, 1998.

\bibitem{Hanamura} M. Hanamura, {\sl Mixed motives and algebraic cycles II\/}, Invent. Math., \textbf{158}, (2004), 105--179.

\bibitem{Hartshorne} R. Hartshorne, {\sl Algebraic Geometry\/}, Graduate Texts in Mathematics, No. 52. Springer-Verlag, New York-Heidelberg, 1977. xvi+496 pp.

\bibitem{Hesselholt Nagoya} L. Hesselholt, {\sl On the $K$-theory of the coordinate axes in the plane\/}, Nagoya Math. J., \textbf{185}, (2007), 93--109.

\bibitem{IR} F. Ivorra, K. R\"ulling, {\sl $K$-groups of reciprocity functors\/}, preprint arXiv:1209.1217, (2012), to appear in J. Algebraic Geom.

\bibitem{Kato} K. Kato, {\sl Milnor $K$-theory and Chow group of zero-cycles\/}, (Boulder, Colo., 1983), Contem. Math., Amer. Math. Soc., Providence, {\bf 55}, (1986), 241--253.  

\bibitem{KS} M. Kerz, S. Saito, {\sl Chow group of 0-cycles with modulus and higher dimensional class field theory\/}, arXiv:1304.4400v1, (2013), to appear in Duke Math. J.

\bibitem{KMM1} J. Koll{\'a}r, Y. Miyaoka, S. Mori, {\sl Rationally connected varieties\/}, J. Algebraic Geom., {\bf 1}, (1992), 429--448.

\bibitem{KMM2} J. Koll{\'a}r, Y. Miyaoka, S. Mori, {\sl Rational connectedness and boundedness of Fano manifolds\/}, J. Differential Geometry, {\bf 36}, (1992), no. 3, 765--779.

\bibitem{Kollar} J. Koll{\'a}r, {\sl Rational Curves on Algebraic Varieties\/}, Ergeb. Math. Grenzgebiete, {\bf 32}, Springer, 1996.  
\bibitem{KL} A. Krishna, M. Levine, {\sl Additive higher Chow groups of schemes\/}, J. Reine Angew. Math., \textbf{619}, (2008), 75--140.

\bibitem{KP} A. Krishna, J. Park, {\sl Moving lemma for additive higher Chow groups}, Algebra Number Theory, \textbf{6}, (2012),  293--326.
\bibitem{KP3} A. Krishna, J. Park, {\sl Mixed motives over $k[t]/(t^{m+1})$},  J. Inst. Math. Jussieu, \textbf{11}, (2012),  611--657. 
\bibitem{KP2} A. Krishna, J. Park, {\sl DGA-structure on additive higher Chow groups\/}, Int. Math. Res. Not., Vol 2015 (2015), no. 1, 1-54. doi: 10.1093/imrn/rnt193

\bibitem{KrSr} A. Krishna, V. Srinivas, {\sl Zero-cycles and $K$-theory on normal surfaces\/}, Ann. Math., {\bf 156}, (2002), 155--195. 

\bibitem{Milnor} J. Milnor, {\sl Algebraic $K$-theory and quadratic forms\/}, Invent. Math., {\bf 9}, (1970), 318--344. 

\bibitem{P1} J. Park, {\sl Algebraic cycles and additive dilogarithm\/}, Int. Math. Res. Not., \textbf{2007}, no. 18, Art. ID rnm067, 19pp. doi:10.1093/imrn/rnm067 

\bibitem{P2} J. Park, {\sl Regulators on additive higher Chow groups\/}, Amer. J. Math., \textbf{131}, (2009), 257--276.

\bibitem{R} K. R\"ulling, {\sl The generalized de Rham-Witt complex over a field is a complex of zero-cycles\/}, J. Algebraic Geom., \textbf{16}, (2007), 109--169.

\bibitem{Scholl} A. J. Scholl, {\sl Classical motives}, in Motives (Seattle, WA, 1991), Proc. Sympos. Pure Math., 55, Part 1, Amer. Math. Soc., Providence, RI, 1994., pp. 163--187.

\bibitem{Totaro} B. Totaro, {\sl Milnor $K$-theory is the simplest part of algebraic $K$-theory\/}, K-Theory, {\bf 6}, (1992), 177--189. 

\end{thebibliography}
\end{document}